\documentclass[a4paper]{amsart}

\usepackage{amsmath, amsthm, amssymb, tikz, tikz-cd, enumerate, setspace, microtype,amscd}
\usepackage[T1]{fontenc}
\usepackage{times}
\usepackage{amsrefs}
\usepackage{hyperref}

\title{A new approach to topological T-duality for principal torus bundles}
\author{Tom Dove and Thomas Schick}
\date{\today}

\address{Mathematisches Institut, Universit\"at G\"ottingen, Bunsenstr. 3-5, 37073 G\"ottingen, Germany}
\email{tomjdove@gmail.com}
\email{thomas.schick@math.uni-goettingen.de}

\theoremstyle{plain}
\newtheorem{theorem}{Theorem}[section]
\newtheorem*{theorem*}{Theorem}
\newtheorem{lemma}[theorem]{Lemma}
\newtheorem{proposition}[theorem]{Proposition}
\newtheorem{corollary}[theorem]{Corollary}

\theoremstyle{definition}
\newtheorem{example}[theorem]{Example}
\newtheorem{definition}[theorem]{Definition}

\newcommand{\C}{\mathbb{C}}
\newcommand{\F}{\mathcal{F}}
\newcommand{\Z}{\mathbb{Z}}

\newcommand{\complexs}{\mathbb{C}}
\DeclareMathOperator{\cyl}{Cyl}
\DeclareMathOperator{\id}{id}
\DeclareMathOperator{\pr}{pr}
\DeclareMathOperator{\im}{im}
\DeclareMathOperator{\thom}{Thom}
\DeclareMathOperator{\Hom}{Hom}

\begin{document}

\onehalfspacing

\begin{abstract}
We introduce a new `Thom class' formulation of topological T-duality for
principal torus bundles. This definition is equivalent to the established one
of Bunke, Rumpf, and Schick but has the virtue of removing the global
assumptions on the H-flux required in the old definition.
With the new definition, we provide easier and more transparent proofs of the
classification of T-duals and generalise the local formulation of T-duality
for circle bundles by
Bunke, Schick, and Spitzweck to the torus case. 
\end{abstract}

\maketitle

\tableofcontents

\section{Introduction}\label{sec:introduction}

Topological T-duality over a base space $X$ is a relationship between two pairs $(E, G)$ and $(\hat E,\hat G)$ consisting of $T^n$-bundles $E,\hat E \to
X$ equipped with gerbes $G \to E$ and $\hat G\to \hat E$.
Such a relation models the underlying topology of T-duality in string theory; $E$ is the background space-time in which strings propagate and $G$ is the H-flux representing the B-field.
T-duality is of interest to topologists due to the T-duality transformation: a T-duality relation between $(E, G)$ and $(\hat E, \hat G)$ comes with an isomorphism (with degree shift) between the twisted K-theory groups of $E$ and $\hat E$.

T-duality is locally described by the Buscher rules, which describe how the metric and connection change under T-duality.
It was known that such transformations change the global topology of the background space-time, but it was not until the introduction of topological T-duality in \cite{BEM} that a systematic method was introduced to determine this change.
It was observed that, under T-duality, there is an exchange of the Chern class with the fiberwise integral of the H-flux.
A formal framework for topological T-duality and the T-duality transformation was then introduced in \cite{BunkeSchick05} for the circle case and \cite{BunkeRumpfSchick} for the torus case.
The former defined T-duality via a Thom class on the $S^3$-bundle associated with $E$ and $\hat E$, whereas the latter took a different approach, using an isomorphism between the two gerbes pulled back to $E \times_X \hat E$.
Included in these papers is a proof of the existence and uniqueness properties of T-duals: in the circle case, a unique T-dual always exists, but in the torus case T-duals may not exist and are rarely unique when they do.
Bunke, Rumpf, and Schick prove an explicit description of the classification of T-duals for principal $T^n$-bundles.

In this paper, we introduce a new `Thom class' definition of T-duality that is a direct generalisation of the Thom class approach for circle bundles in \cite{BunkeSchick05} and has significant advantages over the definition made in \cite{BunkeRumpfSchick}.
Firstly, it requires no assumptions on the pair $(E, G)$. 
This fact is used to show that the global assumptions made on the H-flux in \cite{BunkeRumpfSchick} are not required.
Instead, they are a consequence of the (new) local definition.
Secondly, with this new definition the classification results in \cite{BunkeRumpfSchick} become more transparent and can be proved directly with basic algebraic topology.

Let us describe the formulation of T-duality in \cite{BunkeRumpfSchick}.
To state the definition, we must assume that the H-fluxes are trivial when restricted to the fibers of their respective torus bundles.
T-duals can exist for pairs where this is not true, but then one obtains a `non-classical' T-dual consisting of a bundle of noncommutative tori \cite{MathaiRosenberg05}.
 Two pairs $(E, G)$ and $(\hat E, \hat G)$ are \emph{T-dual} if they belong to a T-duality diagram or equivalently a T-duality triple, defined as follows:

\begin{definition}\label{defn}
A \emph{T-duality diagram} over a space $X$ is a diagram of the following form:
\begin{equation}\label{Tdiagram}
\begin{tikzcd}[column sep={4em,between origins},row sep=1em]
 & p^*G \arrow[ld] \arrow[rd] \arrow[rr, "u"] & & \hat p^* \hat G \arrow[ld] \arrow[rd] & \\
G \arrow[rd] & & E \times_X \hat E \arrow[ld, "p", swap] \arrow[rd, "\hat p"] & & \hat G \arrow[ld] \\
 & E \arrow[rd, "\pi", swap] & & \hat E \arrow[ld, "\hat \pi"] & \\
 & & X &&
\end{tikzcd}
\end{equation}
This must satisfy the \emph{Poincar\'e bundle condition}: For each $x \in X$, the automorphism
	\[
	u(x) = \big( 0 \xrightarrow{ p^* v^{-1}} p^*G|_{E_x} \xrightarrow{ u|_{E_x \times \hat E_x}} \hat p^* \hat G|_{\hat E_x} \xrightarrow{ \hat p^* \hat v} 0 \bigr),
	\]
	 where $v \colon G|_{E_x} \to 0$ and $\hat v \colon \hat G|_{\hat E_x} \to 0$ are chosen trivialisations, gives the following specified class:
	\[
		\bigl[u(x) \bigr] = \Bigl[ \sum_{i=1}^n y_i \cup \hat y_i \Bigr] \in H^2(E_x \times \hat E_x) / H^2(E_x) \oplus H^2(\hat E_x) .
	\]
    The quotient is taken so that the class is independent of the chosen
    trivialisations. 
    The triple $\bigl( (E, G), (\hat E, \hat G), u \bigr)$ is called a \emph{T-duality triple}, and the two pairs $(E, G)$ and $(\hat E, \hat G)$ are \emph{T-dual} if they belong to a T-duality triple. 
\end{definition}

This definition relies on the fact that the group of automorphisms of the
trivial gerbe on a space $X$ is canonically isomorphic to $H^2(X)$.
More generally, the set of morphisms between any two gerbes on $X$ forms a $H^2(X)$-torsor.
T-duality can be formulated using other models for the H-flux instead of gerbes, for example principal $PU(\mathcal H)$-bundles.
These models all share analogous relationships with degree 2 and degree 3 integral cohomology; see the axiomatic notion of twists in \cite{BunkeSchick05}*{\textsection 3}.

Definition \ref{defn} is the same as in \cite{BunkeRumpfSchick} except that they make a further assumption on the Dixmier-Douady classes of the gerbes. The degree 3 cohomology of $E$ has a filtration
\[
0 \subseteq \F^3H^3(E) \subseteq \F^2 H^3(E) \subseteq \F^1 H^3(E) \subseteq H^3(E)
\]
related to the Leray-Serre spectral sequence for $\pi \colon E \to B$: $\F^j
H^3(E)$ is the kernel of the pullback map to $E^{(j-1)}$, where $E^{(j)}$ is the
pullback of $E$ to the $j$-skeleton of $X$ (provided $X$ is a CW-complex).
The group $H^3(\hat E)$ has the corresponding filtration. 
The extra assumption imposed on the H-flux in \cite{BunkeRumpfSchick} is that the classes $[G]$ and $[\hat G]$ belong to $\F^2H^3$ and their leading parts satisfy
\begin{equation}\label{condition}
	[G]^{2,1} = \Bigl[ \sum_{i=1}^n y_i \otimes \hat c_i\Bigr] \in {}^{\pi} E^{2,1}_\infty
	\quad \text{and} \quad
	[\hat G]^{2,1} = \Bigl[ \sum_{i=1}^n \hat y_i \otimes c_i \Bigr] \in {}^{\hat \pi} E^{2,1}_\infty.
\end{equation}
Here, $c_i$ and $\hat c_i$ denote the Chern classes of the bundles $E$ and $\hat E$, respectively, and $y_i$ and $\hat y_i$ are generators of the cohomology of the fiber of these bundles.
The groups ${}^{\pi} E^{2,1}_\infty=\F^2H^3(E)/\F^3H^3(E)$ and ${}^{\hat\pi}
E^{2,1}_\infty$ belong to the spectral sequence for the two torus bundles, and `leading parts' refers to the image of $[G]$ and $[\hat G]$ in these two groups. 

In the T-duality of circle bundles one has that $\pi_!\bigl([G]\bigr) = \hat c$ and  $\hat\pi_!\bigl([\hat G]\bigr) = c$, where $c$ and $\hat c$ are the Chern classes of $E$ and $\hat E$, respectively \cites{BEM,BunkeSchick05}.
Condition \eqref{condition} is a generalisation of this to higher dimensions; the H-flux carries information about the Chern classes of the dual bundle.
Moreover, the existence of different representatives of a class in ${}^\pi E^{2,1}_\infty$ is related to the non-uniqueness of the T-duals of $(E, G)$, see \cite{BunkeRumpfSchick}*{\textsection 2.20}.

In the next section, we introduce an equivalent formulation of Definition \ref{defn} that is a direct generalisation of the Thom class definition of T-duality in \cite{BunkeSchick05}.
With this definition, we show that these assumptions on the H-flux are not necessary.
Specifically, in Section \ref{sec:globalcondition} we show that the H-fluxes
belonging to a T-duality triple automatically satisfy \eqref{condition}.

In the case of principal $S^1$-bundles, the T-dual of a pair $(E, G)$ always exists and is unique; it is characterised by the relations $\pi_!\bigl([G]\bigr) = \hat c_1$ and $\hat \pi_!\bigl([\hat G]\bigr) = c_1$.
For torus bundles of higher dimension, this is no longer true.
T-duals need not exist and if they exist they are likely not unique.
In addition to their formulation of T-duality, the authors of \cite{BunkeRumpfSchick} provide precise existence and uniqueness results, detailing criteria for when T-duals exist and a description of what the T-duals are.
Their proof is based on the construction of a classifying space for T-duality triples and a universal T-duality diagram.
With the new definition introduced in this paper, we give alternative proofs
of these results that come directly from the definition of T-duality and only use basic algebraic topology.

In particular, concerning existence we show:

\begin{theorem*}[Theorem \ref{thm:existence}]
A pair $(E, G)$ over a CW-complex $X$ has a T-dual if and only if $[G] \in \F^2 H^3(E)$.
\end{theorem*}

Regarding (non)-uniqueness of the dual bundle we have in particular:

\begin{theorem*}[Theorem \ref{thm:uniqueness1}]
If $(E, G)$ and $(\hat E, \hat G)$ are T-dual $T^n$-bundles over $X$, then for every antisymmetric matrix $B \in \operatorname{Mat}(n,n,\Z)$ there exists a pair $(\hat E', \hat G')$ with $\hat c_i' = \hat c_i + \sum_{j=1}^n B_{ij}c_j$ that is also T-dual to $(E, G)$.
Moreover, every T-dual of $(E, G)$ satisfies this for some $B$.
\end{theorem*}

We prove the existence results in Section \ref{sec:existence} and the uniqueness results in Section \ref{sec:uniqueness}.

 The Poincar\'e bundle condition in Definition \ref{defn} is a fiberwise condition, and so the removal of the extraneous global assumptions on the H-flux allows a purely local definition of T-duality.
 We demonstrate this in Section \ref{sec:localtduality} by generalising the local formulation of T-duality in \cite{BSS} to principal torus bundles.
 This definition in particular highlights that T-duality for $T^n$-bundles is locally an $n$-fold product of T-duality for circle bundles. 

 Being an important aspect of T-duality, we discuss the T-duality
 transformation in Section \ref{sec:Ttransformation}. To illustrate this, we
 provide some interesting examples and applications of the T-duality
 isomorphism. Our local perspective on T-duality also is used in the author's
 proof that the T-duality transformation is an isomorphism in equivariant
 K-theory for arbitrary actions of a compact group \cite{DoveSchick}.

\section{A Thom class formulation of T-Duality}
\label{sec:thomclassdef}

The main observation for our new approach to T-duality is that there is a bijection between the gerbe isomorphisms $p^* G \to \hat p^* \hat G$ and certain gerbes on the fiberwise join of and $\hat E$, which can be explicitly constructed as
\[
J := \cyl(p) \cup_{E \times_X \hat E} \cyl(\hat p).
\]
In words: we construct the mapping cylinders of the two projections $p \colon E \times_X \hat E \to E$ and $\hat p \colon E \times_X \hat E \to \hat E$ and glue them along the shared copy of $E \times_X \hat E$.
The resulting space can be thought of as a `cylinder' in which each interior slice is a copy of $E \times_X \hat E$ and the ends are copies of $E$ and $\hat E$.
It is a fiber bundle over $X$ whose typical fiber is the join $T^n * T^n$.
The following diagram commutes up to homotopy:
\begin{equation}\label{inclusiondiagram}
\begin{tikzcd}
 &  J  & \\
E \arrow[ru, hook] & E \times_X \hat E \arrow[l, "p"] \arrow[u, "i", hook] \arrow[r, "\hat p", swap]  & \hat E \arrow[lu, hook, swap]
\end{tikzcd}
\end{equation}
The homotopies are those that pull the central $E \times_X \hat E$ slice to either end of the cylinder.
We say that a gerbe $G_J \to J$ restricts to $G$ and $\hat G$ if the pullbacks of $G_J$ to $E$ and $\hat E$ are isomorphic to $G$ and $\hat G$, respectively.

\begin{proposition}\label{prop:bijection}
Let $G \to E$ and $\hat G \to \hat E$ be gerbes.
There is a bijection between the set of gerbe isomorphisms $p^* G \to \hat p^* \hat G$ and the set of isomorphism classes of gerbes on $J$ that restrict to $G$ and $\hat G$. 
\end{proposition}

\begin{proof}
We have canonical maps $\cyl(p) \to E$ and $\cyl(\hat p) \to \hat E$. An isomorphism $u \colon p^* G \to \hat p^* \hat G$ can be used to glue the pullbacks of $G$ and $\hat G$ together along these canonical maps.
These pullbacks restrict to $G$ and $\hat G$ on the end that has not been glued, so we get a suitable gerbe on $J$.

For the other direction, let $G_J$ be a gerbe on $J$ that restricts to $G$ and $\hat G$.
The homotopies in diagram \eqref{inclusiondiagram} give two isomorphisms of gerbes over $E \times_X \hat E$: one between $p^*G$ and $i^*G_J$ and another between $i^*G_J$ and $\hat p^*\hat G$.
Composing these gives the isomorphism we are after. 

One checks that these constructions are inverse to each other (up to canonical
and natural isomorphism).
\end{proof}

With this proposition, we can consider a new kind of T-duality diagram that contains the same information as \eqref{Tdiagram}:
\begin{equation}\label{JTdiagram}
\begin{tikzcd}[column sep={5em,between origins},row sep=2em]
 & G_J \arrow[d] & \\
G \arrow[ru, hook] \arrow[d] & J & \hat G \arrow[d] \arrow[lu, hook]\\
E \arrow[ru, hook] \arrow[rd] & E \times_X \hat E \arrow[u, hook] \arrow[l] \arrow[r] \arrow[d] & \hat E \arrow[ld] \arrow[lu,  hook] \\
 & X & 
\end{tikzcd}
\end{equation}
There is now the question of how to interpret the Poincar\'e bundle condition in this context.
This will be expressed as a condition on the characteristic class of the gerbe $G_J$. 

The join $T^n * T^n$ can be decomposed into two open sets that each
deformation retract onto the two copies of $T^n$ that are joined and whose intersection is homotopy equivalent to $T^n \times T^n$.
The corresponding Mayer-Vietoris sequence is
\begin{equation}\label{MayerVietoris}
\dotsm \to H^k(T^n) \oplus H^k(T^n) \to H^k(T^n \times T^n) \to H^{k+1}(T^n * T^n) \to \dotsm.
\end{equation}
The map $H^k(T^n) \oplus H^k(T^n) \to H^k(T^n \times T^n)$ is injective for
all $k>0$, implying that we then have exact sequences
\[
0 \to H^k(T^n) \oplus H^k(T^n) \to H^k(T^n \times T^n) \to H^{k+1}(T^n * T^n) \to 0.
\]
In particular,
\begin{equation}\label{eqn:cohomtorusjoin}
H^3(T^n * T^n) \cong H^2(T^n \times T^n) / H^2(T^n) \oplus H^2(T^n).
\end{equation}
Notice that this group contains the class $[u(x)]$ introduced in Definition \ref{defn}.
For $k \in \{0,1\}$ the map $H^k(T^n) \oplus H^k(T^n) \to H^k(T^n \times T^n)$ is also surjective, implying that 
\[
H^1(T^n * T^n) = H^2(T^n * T^n) = 0.
\]
Using the Leray-Serre spectral sequence of $J \to X$ we then calculate that
the projection induces isomorphisms $H^1(J) \cong H^1(X)$, $H^2(J) \cong
H^2(X)$, and that $H^3(J)$ fits into an exact sequence
\begin{equation}\label{exactsequenceJ}
0 \to H^3(X) \to H^3(J) \xrightarrow{\iota^*} H^3(T^n * T^n) \to H^4(X) \to 0.
\end{equation}
The center map $\iota^*$, which is the restriction to the fiber of $J \to X$, provides the connection between the local and global conditions.

\begin{definition}
A Thom class for J is a class in $H^3(J)$ that maps to 
\[
\Bigl[ \sum_{i=1}^n y_i \cup \hat y_i \Bigr] \in H^3(T^n * T^n)
\]
under the homomorphism $\iota^*$ in \eqref{exactsequenceJ}, where $y_i$ and $\hat y_i$ are the generators of $H^2(T^n \times T^n)$ and we have made the identification \eqref{eqn:cohomtorusjoin}.
The set of Thom classes for $J$ will be denoted by $\thom(J)$. Note that
\eqref{exactsequenceJ} implies that $\thom(J)$ is an $H^3(X)$-torsor if it is non-empty.
\end{definition}

We are now prepared for a new notion of topological T-duality.

\begin{definition}\label{def:ThomTduality}
A \emph{T-duality triple} is a triple $\bigl((E, G), (\hat E, \hat G), G_J \bigr)$ where $G_J$ is a gerbe on $J$ that restricts to $G$ and $\hat G$ and such that $[G_J]$ is a Thom class for $J$. 
The pairs $(E, G)$ and $(\hat E, \hat G)$ are T-dual if they belong to a T-duality triple.
\end{definition}

Let us explain how this is a direct generalisation of the definition of T-duality for circle bundles in \cite{BunkeSchick05}.
For $n=1$, $J$ is an $S^1 * S^1 \cong S^3$-bundle over $X$.
It is the sphere bundle of the vector bundle $L \oplus \hat L$, where $L$ and $\hat L$ are the line bundles associated with $E$ and $\hat E$, respectively.
In this case, a Thom class for $J$ is an honest Thom class because $y \cup
\hat y$ is a generator of $H^3(S^3)\cong H^2(S^1\times S^1)$.

Before showing that this new definition is equivalent to the previous one, we observe that, unlike the previous definition, we did not need to assume that the H-fluxes are trivialisable when restricted to the $T^n$-fibers.
It instead follows from the definition:

\begin{theorem}\label{thm:F1page}
Let $\bigl((E, G), (\hat E, \hat G), G_J \bigr)$ be a T-duality triple. 
Then $G$ (resp. $\hat G$) is trivialisable when restricted to the fiber of $E \to X$ (resp. $\hat E \to X$).
\end{theorem}

\begin{proof}
We have the following commutative diagram induced by inclusion maps:
\[
\begin{tikzcd}
    H^3(J)
        \rar \dar
    & H^3(E)
        \dar
    \\
    H^3(T^n * T^n)
        \rar
    & H^3(T^n)
\end{tikzcd}
\]
By considering the Mayer-Vietoris sequence \eqref{MayerVietoris}, and again noting that the map $H^k(T^n) \oplus H^k(T^n) \to H^k(T^n \times T^n)$ is injective, we conclude that the lower horizontal map is the zero.
Thus, the class $[G]$, which is the image of $[G_J]$, is zero when restricted to the fiber $T^n$.
\end{proof}

Now, let us show that the two definitions are equivalent.

\begin{theorem}
Definition \ref{defn} and Definition \ref{def:ThomTduality} are equivalent definitions of T-duality.
\end{theorem}

\begin{proof}
Consider the following:
\begin{center}
\begin{tikzcd}
    \{\text{Automorphisms of the trivial gerbe over $T^n \times T^n$}\} 
        \arrow[d, <->, "\text{Proposition \ref{prop:bijection}}", swap] \arrow[r,"\cong"] 
    &  H^2(T^n \times T^n) 
        \arrow[d, "\eqref{eqn:cohomtorusjoin}"] 
    \\
    \{\text{Gerbes on $T^n * T^n$ restricting to $G|_{T^n}$ and $\hat G|_{T^n}$}\} 
        \arrow[r]
    & H^3(T^n * T^n)
    \\
    \{\text{Gerbes on $J$ restricting to $G$ and $\hat G$}\}
        \arrow[u, "\text{Restrict to fiber}"] \arrow[r]
    & H^3(J) 
        \arrow[u, "\eqref{exactsequenceJ}", swap]
\end{tikzcd}
\end{center}
The double-sided vertical arrow is a bijection that relies on a choice of trivialisation of $G|_{T^n}$ and $\hat G|_{T^n}$ together with Proposition \ref{prop:bijection}.
The upper horizontal map uses that automorphisms of the trivial gerbe are classified by degree 2 cohomology.
The remaining horizontal maps send a gerbe to its Dixmier-Douady class.
The two vertical maps on the right-hand side are labelled with the previously discussed equations they come from.
If this diagram commutes, then the result is proved:  $G_J$ in the bottom left is sent to $u$ in the top left, so that if the diagram commutes then they are both mapped to the same element in $H^3(T^n * T^n)$.
Then $[G_J]$ is sent to $\bigl[\sum_{i=1}^n y_i \cup \hat y_i \bigr] \in H^3(T^n * T^n)$ if and only if $u$ is as well.

So, let us show that the diagram commutes.
The vertical maps in the lower square are induced by inclusions and these commute with the horizontal characteristic class maps.
Therefore the bottom square commutes and it remains to check that the upper square commutes.
After trivialisations are chosen, this requires showing that the following commutes:
\begin{center}
\begin{tikzcd}
    \{\text{Gerbe isomorphisms $0 \to 0$ over $T^n \times T^n$}\} 
        \arrow[d, <->,"\cong"] \arrow[r, <->,"\cong"] 
    &  H^2(T^n \times T^n) 
        \arrow[d] 
    \\
    \{ \text{Gerbes on $T^n * T^n$ restricting to 0 on both ends} \} 
        \arrow[r] 
    & H^3(T^n * T^n) 
\end{tikzcd}
\end{center}
We can identify the image of the lower map with $H^3(S(T^n \times T^n))$, where $S$ denotes the suspension. 
This is because $S(T^n \times T^n)$ is obtained from $T^n \times T^n$ by collapsing the ends to points, so gerbes on $S(T^n \times T^n)$ are equivalent to gerbes on $T^n * T^n$ that are trivial on each end.
We also have the following commutative diagram:
\begin{center}
\begin{tikzcd}
    \{\text{Gerbe isomorphisms $0 \to 0$ over $T^n \times T^n$}\} 
        \arrow[d, <->] \arrow[r, <->] 
    &  H^2(T^n \times T^n) 
        \arrow[d] 
    \\
    \{ \text{Gerbes on $S(T^n \times T^n)$} \}
        \arrow[r] 
    & H^3(S(T^n \times T^n)) 
\end{tikzcd}
\end{center}
The arrow on the right side is the suspension isomorphism, which is also the boundary map in the Mayer-Vietoris sequence for the suspension.
The diagram commutes by general properties of twists, see \cite{BunkeSchick05}*{\textsection 3.1.2 (3)}.
Using naturality of the Mayer-Vietoris sequence along the collapse map $T^n * T^n \to S(T^n \times T^n)$, we get that the following commutes:
\[
\begin{tikzcd}
    H^2(T^n \times T^n) 
        \arrow[r, "="] \arrow[d]
    & H^2(T^n \times T^n) 
        \arrow[d] 
    \\
    H^3(S(T^n \times T^n))
    \arrow[r]
    & H^3(T^n * T^n).
\end{tikzcd}
\]
The result is then obtained by putting the last two diagrams together and
identifying the set of gerbes on $S(T^n \times T^n)$ with the gerbes on $T^n *
T^n$ that restrict to zero at the two ends.

The choices of trivialisations of $G|_{T^n}$ and $\hat G|_{T^n}$ each form a
$H^2(T^2)$-torsor. Different choices of trivialisation will therefore produce
classes in $H^2(T^n \times T^n)$ that differ by an element in the image of
$H^2(T^n) \oplus H^2(T^n) \to H^2(T^n \times T^n)$. The map $H^2(T^n \times
T^n) \to H^3(T^n * T^n)$ is precisely the quotient by this image, so different choices of trivialisation produce the same element in $H^3(T^n * T^n)$. Thus the commutativity of the required diagram is independent of any choice of trivialisation, and the proof is complete.
\end{proof}

This implies the following corollary, which is mentioned in \cite{BunkeRumpfSchick} without an explicit proof:

\begin{corollary}
The definition of T-duality in \cite{BunkeSchick05} is equivalent to the definition in \cite{BunkeRumpfSchick} restricted to $n=1$.
\end{corollary}

\begin{proof}
This follows from the previous theorem since Definition \ref{def:ThomTduality} is equivalent to the definition in \cite{BunkeSchick05} when $n=1$.
\end{proof}

\section{Existence of T-duals}
\label{sec:existence}

With the Thom class definition of T-duality, it is easy to determine when two bundles $E$ and $\hat E$ can belong to a T-duality triple.
The first part of the following result is not explicitly stated in \cite{BunkeRumpfSchick}, but follows from their construction of the classifying space for T-duality triples.
The second part is essentially \cite{BunkeRumpfSchick}*{Proposition 7.4}.

\begin{theorem}\label{thm:existencetwobundles}
Let $\pi \colon E \to X$ and $ \hat \pi \colon \hat E \to X$ be principal $T^n$-bundles.
There exists gerbes $G \to E$ and $\hat G \to \hat E$ such that $(E ,G)$ and $(\hat E, \hat G)$ are T-dual if and only if
\[
    \sum_{i=1}^n c_i \cup \hat c_i = 0 \in H^4(X),
\]
where $\{ c_i \}_{i=1}^n$ and $\{\hat c_i\}_{i=1}^n$ are the Chern classes classifying $E$ and $\hat E$, respectively.
Moreover, if such a dual pair exists, then any other must be of the form $(E, G \otimes \pi^*Q)$ and $(\hat E, \hat G \otimes \hat \pi^*Q)$ for some gerbe $Q \to X$.
\end{theorem}

\begin{proof}
Given the two principal $T^n$-bundles, we can form their fiberwise join $J$.
Its third cohomology group fits into \eqref{exactsequenceJ}:
\[
    0 \to H^3(X) \to H^3(J) \to H^3(T^n * T^n) \to H^4(X) \to 0.
\]
Using the identification $H^3(T^n*T^n)= H^2(T^n\times T^n)/H^2(T^n)\oplus
H^2(T^n)$, the final map in this sequence is given by
\[
    \Bigl[ \sum_{i,j} a_{i,j}\, y_i \cup \hat y_j \Bigr]
    \longmapsto \sum_{i,j} a_{ij}\, c_i \cup \hat c_j,
\]
where $\{ y_i \}$ and $\{\hat y_i \}$ are the generators of each copy of $H^1(T^n)$.
Therefore, since $\bigl[\sum_{i=1}^n y_i \cup \hat y_i \bigr]$ is mapped to $\sum_{i=1}^n c_i \cup \hat c_i$, there exists a Thom class for $J$ if and only if $\sum_{i=1}^n c_i \cup \hat c_i = 0$.
Given such a class in $H^3(J)$, we can take the corresponding gerbe and restrict it to $E$ and $\hat E$ to get the required gerbes in the T-duality triple.

The set of Thom classes is an $H^3(X)$-torsor, so given any Thom class we can pull back a gerbe on $X$ to get another, and any two Thom classes are related by a unique element of $H^3(X)$.
This implies the second statement of the theorem.
\end{proof}

We now head towards the main existence result.
The question we want to answer is whether a given pair $(E, G)$ has a T-dual.
More specifically, we ask whether there exists a T-duality triple containing $(E, G)$.
In \cite{BunkeRumpfSchick}, the authors introduce the notion of an extension:

\begin{definition}
An extension of a pair $(E, G)$ is a T-duality triple of the form
\begin{equation*}
\bigl( (E, G), (\hat E, \hat G), G_J \bigr).
\end{equation*}

\end{definition}

The main existence result states that there exists an extension of $(E, G)$ if and only if $[G] \in \F^2H^3(E)$.
This was first proved in \cite{BunkeRumpfSchick}*{Theorem 2.23} by constructing a classifying space for T-duals.
Our proof starts with two lemmas, the first being the forward direction of the proof.

\begin{lemma}\label{lem:classforward}
If $(E, G)$ has an extension to a T-duality triple, then $[G] \in \F^2H^3(E)$.
\end{lemma}

\begin{proof}
If $(E,G)$ belongs to a T-duality triple then Theorem \ref{thm:F1page} implies that $[G] \in \F^1H^3(E)$.
The result then follows from \cite{schneiderthesis}*{Lemma 3.7}, which says that being in the first filtration step is equivalent to being in the second filtration steps for pairs fitting into a T-duality diagram.
For convenience, we will write out the argument.

To show that $[G] \in \F^2H^3(E)$ we need that if $f \colon C \to X$ is a map
from a 1-dimensional CW-complex $C$, then $F^*[G] = 0$, where $F \colon f^*E
\to E$ is the canonical map covering $f$.
Let $(\hat E, \hat G)$ be a T-dual of $(E,G)$.
Then there exists an isomorphism $p^*G \cong \hat p^*\hat G$ where $p$ and $\hat p$ are the projections from $E \times_X \hat E$ to $E$ and $\hat E$, respectively.
Since $H^2(C) = 0$, the bundle $f^*\hat E$ is trivial and there exists a section $\hat\sigma \colon C \to f^*\hat E$.
Consider the following diagram:
\[
\begin{tikzcd}
    && f^*E 
        \arrow[ld, "{(\id, \hat\sigma \circ \pi)}", swap] 
        \arrow[rd, "\pi"] 
    & \\
    & f^*E \times_C \hat f^*\hat E
        \arrow[ld, "p_C", swap] \arrow[rd, "\hat p_C"]
    && C 
        \arrow[ld, "\hat\sigma"]\\
    f^*E
    && \hat f^*\hat E &
\end{tikzcd}
\]
The maps $p_C$ and $\hat p_C$ are projection maps and $\pi$ is the bundle projection map.
Let $\hat F \colon f^*\hat E \to \hat E$ be the map defined in the pullback of $\hat E$ along $f$.
We calculate:
\begin{align*}
    F^*G
        &\cong (\id, \hat\sigma \circ \pi)^*p_C^*F^*G \\
        &\cong (\id, \hat\sigma \circ \pi)^*(F,\hat F)^*p^*G \\
        &\cong (\id, \hat\sigma \circ \pi)^* (F,\hat F)^* \hat p^*\hat G \\
        &\cong (\id, \hat\sigma \circ \pi)^* \hat p_C^* \hat F^*\hat G \\
        &\cong \pi^* \hat\sigma^* \hat F^* \hat G \\
        &\cong 0
\end{align*}
Let us explain the steps.
First, we use that $p_C \circ (\id, \hat\sigma \circ \pi) = \id$.
The next three steps use that $(F,\hat F) \colon f^*E \times_C f^*\hat E \to E \times_X \hat E$ commutes with the projection maps and that there is an isomorphism $p^*G \cong \hat p^* \hat G$.
In the second last step, we use that $\hat p_C \circ (\id, \hat\sigma \circ \pi) = \hat\sigma \circ \pi$.
For the final step, we note that $H^3(C) = 0$, so that $\hat\sigma^*\hat F^*\hat G$ is trivial.
This completes the proof.
\end{proof}

We establish some notation for the next lemma.
If $E \to X$ is a principal $T^n$-bundle with Chern classes $c_1, \dotsc, c_n \in H^2(X)$, then 
\[
    E \cong E_1 \times_X E_2 \times_X \dotsm \times_X E_n,
\]
where each $E_i$ is a principal $S^1$-bundle classified by $c_i$.
If $\hat E$ is a second bundle with the decomposition $\hat E \cong \hat E_1 \times_X \hat E_2 \times_X \dotsm \times_X \hat E_n$, then we can write $J$ as 
\[
    J \cong J_1 \times_{X \times I} \times \dotsm \times_{X \times I} J_n,
\]
where each $J_i$ is the fiberwise join of $E_i$ and $\hat E_i$.
This decomposition will allow us to restrict to the $n=1$ case.

Note that if $[G_J] \in \thom(J)$, then the corresponding gerbe belongs to a T-duality triple connecting $(E, G_J|_E)$ and $(\hat E, G_J|_{\hat E})$.
Therefore, Lemma \ref{lem:classforward} implies that $[G_J|_E] \in \F^2H^3(E)$.
This means that the restriction of a Thom class to $E$ lies in $\F^2H^3(E)$, or in other words, restriction induces a map $\thom(J) \to \F^2H^3(E)$.

The following lemma can be considered a preliminary version of the existence theorem.
It says that if we have a pair $(E, G)$ and the dual bundle $\hat E$ satisfying the right relation with $[G]^{2,1}$, then we can find the dual gerbe.

\begin{lemma}\label{restlemma}
Let $(E, G)$ be a pair over a CW-complex $X$ with $[G] \in \F^2 H^3(E)$ and let $\hat E$ be a
$T^n$-bundle with Chern classes $\hat c_1, \dotsc, \hat c_n \in H^2(X)$. Let
$J$ be the fiberwise join of $E$ and $\hat E$.
If $[G]^{2,1} = [\sum_{i=1}^n y_i \otimes \hat c_i ]$, then there exists a Thom class $[G_J]$ that restricts to $[G]$ under the inclusion $E \hookrightarrow J$.
\end{lemma}

\begin{proof}
The proof has two steps: first, we discuss the case where $X$ is
two-dimensional and then we use this to generalise to the arbitrary setup.

\textbf{Step 1:} Assume that $X$ is a two-dimensional CW-complex.
In this case, the exact sequence \eqref{exactsequenceJ} implies that $H^3(J) \cong H^3(T^n * T^n)$, so that there is a unique Thom class $[G_J] \in \thom(J)$.

Using the notation established before the statement of the lemma, on each $E_i$ there is a gerbe $G_i$ such that $[G_i]^{2,1} = y \otimes \hat c_i$.
We then have
\[
    [G] = \sum_{i=1}^n \pr_i^*[G_i].
\]
For each pair $(E_i, G_i)$, there is a unique Thom class $[G_J^{(i)}] \in \thom(J_i)$ such that $[G_J^{(i)}|_E] = [G_i]$.
Being a Thom class, $[G_J^{(i)}]$ maps to $[y \cup \hat y]$ under the isomorphism $H^3(J_i) \cong H^3(S^1 * S^1)$.
This means that $\sum_{i=1}^n \pr_i^*[G_J^{(i)}] \in H^3(J)$ is mapped to $\bigl[\sum_{i=1}^n y_i \cup \hat y \bigr]$ under the isomorphism with $H^3(T^n * T^n)$, that is, it is a Thom class.
By uniqueness of the Thom class we must then have $[G_J] = \sum_{i=1}^n p_i^*[G_J^{(i)}]$.
Therefore,
\[
   [G_J|_E] 
   = \sum_{i=1}^n [G_J^{(i)}|_E]
   = \sum_{i=1}^n [G_i]
   = [G],
\]
which implies the result in the case where $X$ is two-dimensional.

\textbf{Step 2: } Now let $X$ be a CW-complex and $X^{(2)}$ its 2-skeleton.
Let $(E^{(2)}, G^{(2)})$ and $J^{(2)}$ be the restriction of $(E, G)$ and $J$ to $X^{(2)}$.
By the first part of the proof, there exists a unique $[G^{(2)}_J] \in \thom(J^{(2)})$ such that $[G^{(2)}_J|_E] = [G^{(2)}]$.
Any $[G_J] \in \thom(J)$ restricts to an element of $\thom(J^{(2)})$, but $[G_J^{(2)}]$ is the only element in $\thom(J^{(2)})$, so all $[G_J]$ restricts to $[G_J^{(2)}]$.

By standard results in cellular cohomology, the inclusion of the 2-skeleton induces an injection $H^2(X) \hookrightarrow H^2(X^{(2)})$.
Writing $\pi^{(2)} \colon E^{(2)} \to X^{(2)}$, this further induces an injection ${}^{\pi}E^{2,1}_\infty \to {}^{\pi^{(2)}}E^{2,1}_\infty$.
We can see this by considering the following diagram:
\[
\begin{tikzcd}
    0 \rar
    & H^2(T^n)
        \rar["{}^{\pi}d^{0,2}_2"] \dar["="]
    & \ker({}^\pi d^{2,1}_2)
        \rar \dar[hook]
    & {}^{\pi}E^{2,1}_\infty
        \dar \rar
    & 0 \\
    0 \rar
    & H^2(T^n)
        \rar["{}^{\pi^{(2)}}d^{0,2}_2"] 
    & \ker({}^{\pi^{(2)}} d^{2,1}_2)
        \rar 
    & {}^{\pi^{(2)}}E^{2,1}_\infty
        \rar
    & 0 
\end{tikzcd}
\]
The top row comes form the spectral sequence for $E$ and the bottom row from the spectral sequence for $E^{(2)}$.
The center vertical map is injective because it is a restriction of $H^1(T^n) \otimes H^2(X) \to H^1(T^n) \otimes H^2(X^{(2)})$, which is injective.
One can check via a simple diagram chase that the first vertical map being an equality and the center map being injective implies that the right vertical map is injective.

So, we now consider the diagram
\[
\begin{tikzcd}
    \thom(J)
        \rar \dar
    & \F^2 H^3(E)
        \dar \rar
    & {}^{\pi}E^{2,1}_\infty
        \dar[hook]
    \\
    \thom(J^{(2)})
        \rar
    & \F^2 H^3(E^{(2)})
        \rar
    & {}^{\pi^{(2)}}E^{2,1}_\infty.
\end{tikzcd}
\]
The class $[G_J]$ in the top left is sent to $\bigl[ \sum_{i=1}^n y_i \otimes \hat c_i \bigr]$ in the bottom right, as seen by taking the path through $\thom(J^{(2)})$ and $\F^2H^3(E^{(2)})$. 
Injectivity of the right-most vertical map then implies that $[G_J]$ is mapped to $\bigl[ \sum_{i=1}^n y_i \otimes \hat c_i \bigr]$ in the top right.
It is possible, however, that $[G_J]$ does not restrict to $[G] \in \F^2H^3(E)$, but we will show that one can modify it by an element in $H^3(X)$ so that it does.
Consider the following short exact sequence coming from the spectral sequence associated with $E$:
\[
    0 \to {}^\pi E^{3,0}_\infty \to \F^2H^3(E) \to {}^{\pi}E^{2,1}_\infty \to 0.
\]
Both $[G_J|_E]$ and $[G]$ map to the same element in ${}^\pi E^{2,1}_\infty$, and therefore  differ by an element coming from ${}^\pi E^{3,0}_\infty$.
The group ${}^\pi E^{3,0}_\infty$ is a quotient of $H^3(X)$, so we can choose a $[Q] \in H^3(X)$ that represents the aforementioned difference element.
Then $[G_J] + j^*[Q]$ is a Thom class representing a gerbe that restricts to $G$, where $j \colon J \to X$ is the projection onto $X$.
This completes the proof.
\end{proof}

Now we prove the main existence result:

\begin{theorem}\label{thm:existence}
A pair $(E, G)$ over a CW-complex $X$ has a T-dual dual if and only if $[G] \in \F^2 H^3(E)$.
\end{theorem}

\begin{proof}

The forward direction of the proof is Lemma \ref{lem:classforward}.
For the reverse direction, assume that $[G] \in \F^2H^3(E)$.
Since ${}^{\pi}E^{2,1}_\infty$ is a sub-quotient of $H^1(T^n) \otimes H^2(X)$, $[G]^{2,1}$ is of the form $\bigl[\sum_{i=1}^n y_i \otimes \hat c_i \bigr]$ for some $\hat c_1, \dotsc, \hat c_n \in H^2(X)$.
Let $\hat E$ be the principal bundle having Chern classes $\hat c_1, \dotsc, \hat c_n \in H^2(X)$. 
Lemma \ref{restlemma} implies that there is a gerbe $\hat G \to \hat E$ such that $(E, G)$ and $(\hat E, \hat G)$ are T-dual.
This completes the proof.
\end{proof}

\begin{corollary}
In the $n=1$ case, T-duals always exist.
\end{corollary}

\begin{proof}
Since $H^3(S^1) = H^2(S^1) = 0$, we always have $H^3(E) = \F^2H^3(E)$.
\end{proof}

\section{Global condition on the H-flux}
\label{sec:globalcondition}

When $n=1$, two T-dual pairs $(E, G)$ and $(\hat E, \hat G)$ are characterised
by the property
\begin{equation}\label{eq:push_relation}
\pi_!\bigl([G]\bigr) = \hat c_1; \quad \hat \pi_!\bigl([\hat
G] \bigr) = c_1.
\end{equation}
This is proven in \cite{BunkeSchick05}*{Lemma 2.29} using the
universal space for T-duality triples. To be self-contained we give an
independent and self-contained proof below; indeed this is a general statement
about Thom classes for direct sums of two complex line bundles which should be
well known.

The following is the generalisation of \eqref{eq:push_relation} to $n \geq 1$.

\begin{theorem}\label{conditiontheorem}
If $(E, G)$ and $(\hat E, \hat G)$ are T-dual pairs over a CW-complex $X$ then
\begin{equation}\label{con}
[G]^{2,1} = \Bigl[ \sum_{i=1}^n y_i \otimes \hat c_i \Bigr] \in {}^{\pi} E^{2,1}_\infty
	\quad \text{and} \quad
	[\hat G]^{2,1} = \Bigl[ \sum_{i=1}^n \hat y_i \otimes c_i \Bigr] \in {}^{\hat \pi} E^{2,1}_\infty.
\end{equation}
\end{theorem}

We prove this in the same way we proved Lemma \ref{restlemma}, but we break
the steps into three lemmas. Note that the argument involves a reduction to the
case $n=1$ which we prove independently as the first lemma.
The second lemma says that the result is true for spaces homotopy equivalent to
a 2-dimensional CW complex and the third says that  the assertion for general
$X$ follows from the one for its 2-skeleton. 

Before proving the lemmas, we show that, in the $n=1$ case, \eqref{con} is equivalent to \eqref{eq:push_relation}.
The push-forward map is part of the Gysin sequence, which is derived from the spectral sequence.
It factors as follows:
\[
\begin{tikzcd}
H^3(E) \arrow[r, "\pi_!"] \arrow[d] & H^2(X) \\
{}^\pi E^{2,1}_{\infty} \arrow[r, hook] &  H^1(T) \otimes H^2(X) \arrow[u, "\sim", swap] 
\end{tikzcd}
\]
We see from this that $[G]$ is mapped to $y \otimes \hat c$ precisely when $\pi_!\bigl([G]\bigr) = \hat c$.
The corresponding statement is true for the dual H-flux.

\begin{lemma}
  Given two $U(1)$-principal bundles $\pi\colon E\to X$, $\hat\pi\colon \hat
  E\to X$ with Chern classes 
  $c,\hat c\in H^2(X)$, let $E_\complexs$ and $\hat E_\complexs$ be
  the associated complex line bundles. Let $S(E_\complexs\oplus \hat
  E_\complexs)$ be the unit sphere bundle of the direct sum of the line
  bundles and assume that $T\in H^3(S(E_\complexs\oplus \hat
  E_\complexs))$ is
  a Thom class. Let $\iota\colon E\to S(E_\complexs\oplus \hat E_\complexs)$ be
  the canonical inclusion and correspondingly $\hat \iota$ the canonical inclusion
  of $\hat E$. Then,
  \begin{equation*}
    \pi_!(\iota^*T) = \hat c
    \quad \text{and} \quad
    \hat\pi_!(\hat \iota^*T)= c.
  \end{equation*}
\end{lemma}

\begin{proof}
By symmetry, it is sufficient to prove that $\pi_!(\iota^*T) = \hat c$.
  We have the canonical inclusion of disk bundles $\iota\colon
   D(E_\complexs)\to D(E_\complexs\oplus \hat E_\complexs)$ into the first
   summand, restricting to the inclusion 
  $\iota\colon E\to S(E_\complexs\oplus \hat E_\complexs)$ of boundaries. By
  naturality of the pair sequence we have
  \begin{equation*}
    \iota^* \partial(T) = \partial(\iota^*T) \in
    H^4(D(E_\complexs),E) \cong H^4_{c}(E_\complexs),
  \end{equation*}
  where we use the canonical isomorphism to the compactly supported
  cohomology of the total space of the complex line bundle. 

  Now observe that, under the corresponding canonical isomorphism,
  \begin{equation*}
    \partial T \in H^4\bigl(D(E_\complexs\oplus \hat
    E_\complexs),S(E_\complexs\oplus \hat E_\complexs)\bigr) \cong
    H^4_c(E_\complexs\oplus \hat E_\complexs)
  \end{equation*}
  is mapped to the Thom class $Th$, which for the direct sum of two
  bundles is the cup product
  \begin{equation*}
Th= \pr_E^*Th_E\cup \pr_{\hat E}^*Th_{\hat E}
\end{equation*}
of the Thom classes in
  compactly supported cohomology $H_c^2(E_\complexs)$ and
  $H_c^2(\hat E_\complexs)$, respectively.
  Here, $\pr_E$ and $\pr_{\hat E}$ denote the projection maps.

Observe also that the integration maps for the circle bundle $E$ and the
associated disk or line bundle are compatible, i.e.~that the following diagram commutes:
%\begin{equation*}
%  \begin{CD}
%    H^*(E) @>\partial>> H^{*+1}_c(E_\complexs)\\
%    @VV{\pi_!}V @VV{(\pi_{E_\complexs})_!}V\\
%    H^{*-1}(X) @= H^{*-1}(X)
%  \end{CD}
%\end{equation*}
\[
\begin{tikzcd}[column sep = 1em]
    H^*(E) \ar[rr,"\partial"] \ar[rd, "\pi_!", swap]
    && H^{*+1}_c(E_{\C}) \ar[ld, "(\pi_{E_\C})_!"] \\
    & H^{*-1}(X) &
\end{tikzcd}
\]
  Also note that we have a commutative diagram of embeddings
\[
\begin{tikzcd}
    E 
        \rar["\iota"] \dar["\pi"]
    & E_{\C} \oplus \hat E_{\C}
        \dar["{\pr_{\hat E}}"] 
    \\
    X 
        \rar["z"]
    & \hat E_{\C}
\end{tikzcd}
\]
  where $z$ is the zero section.
  Now we compute, using $\iota$ for the different inclusions associated with
  the inclusion $\iota\colon E_\complexs\to E_\complexs\oplus \hat
  E_\complexs$, $v\mapsto (v,0)$:
    \begin{multline*}
      \pi_!(\iota^*T) 
      = (\pi_{E_\complexs})_!(\partial(\iota^*T))
      = (\pi_{E_\complexs})_!(\iota^*\partial(T)) \\
      = (\pi_{E_\complexs})_!(\iota^*Th) 
      = (\pi_{E_\complexs})_!(\iota^*(\pr_E^*Th_E\cup \pr_{\hat E}^*Th_{\hat E}))  \\
      = z^*((\pr_{\hat E})_!(\pr_E^*Th_E\cup \pr_{\hat E}^*Th_{\hat E}))
      = z^*Th_{\hat E} =\hat c.
    \end{multline*}
Here, we have used that the push forward of the Thom class $\pr_E^*Th_E$ of the
  complex line bundle $\pr_{\hat E}\colon E_\complexs\oplus \hat
  E_\complexs\to \hat E_\complexs$ is $1$, the pull-push
  formula ``$p_!(a\cup p^*b)= (p_!a)\cup b$'' and the fact that the Chern class is the restriction to the zero
  section of the Thom class.
\end{proof}
Note that this implies \eqref{eq:push_relation} as a special case by the Thom
class definition of T-duality triples.

\begin{lemma}
Let $X$ have the homotopy type of a 2-dimensional CW-complex.
If $(E, G)$ and $(\hat E, \hat G)$ are T-dual pairs over $X$ then condition \eqref{con} holds.
\end{lemma}
\begin{proof}
We only prove the condition on $[G]$; symmetry gives the result for $[\hat G]$ as well.
Start by writing $E$ and $\hat E$ as a fiber product of principal circle bundles,
\[
E \cong E_1 \times_X \dotsm \times_X E_n,
\quad
\hat E \cong \hat E_1 \times_X \dotsm \times_X \hat E_n.
\]
Then $J \cong J_1 \times_{X \times I} \dotsm \times_{X \times I} J_n$, where each $J_i$ is the fiberwise join of $E_i$ and $\hat E_i$. 
Let $p_i \colon J \to J_i$ be the projection onto the $i$th factor. Consider the exact sequence
\[
0 \to H^3(X) \to H^3(J) \to H^3(T^n * T^n) \to H^4(X) \to 0.
\]
As $X$ is two-dimensional, $H^3(X) = H^4(X) = 0$, so the map $H^3(J) \to H^3(T^n * T^n)$ is an isomorphism. The class $[G_J]$ corresponds to $\bigl[\sum_{i=1}^n y_i \cup \hat y_i \bigr]$ under this isomorphism. For each $i$ we have
\begin{align*}
H^3(J_i) \cong H^3(S^1* S^1) &\xrightarrow{\,\, p^* \,\,} H^3(T^n * T^n) \cong H^3(J) \\
[y \cup \hat y] &\mapsto [y_i \cup \hat y_i].
\end{align*}
Therefore, since $[G_J]$ is a sum of such classes, there exists gerbes $G_1, \dotsc, G_n$ over $J_1, \dotsc, J_n$, respectively, such that
\[
G_J \cong p_1^*G_1 \otimes \dotsm \otimes p_n^*G_n.
\]
Moreover, each $G_i$ induces a T-duality relation between the pairs $(E_i, G_i|_{E_i})$ and $(\hat E_i, G_i|_{\hat E_i})$. The following commutes:
\[
\begin{tikzcd}
H^3(J_i) \arrow[r] \arrow[d] & H^3(E_i) \arrow[d] \arrow[r] & {}^{\pi_i}E^{2,1}_\infty \arrow[d] \\
H^3(J) \arrow[r]  &  H^3(E) \arrow[r] & {}^\pi E^{2,1}_\infty
\end{tikzcd}
\]
Along the top row, $[y \cup \hat y]$ is mapped to $[y \otimes \hat c_i]$ since we know the result to be true for circle T-duality. Then commutativity of the diagram shows that, in the bottom row, $[y_i \cup \hat y_i]$ is mapped to $[y_i \otimes \hat c_i]$. These maps are all group homomorphisms, so additivity implies the result.
\end{proof}

\begin{lemma}
Let $(E, G)$ and $(\hat E, \hat G)$ be T-dual pairs over a CW complex $X$. If \eqref{con} holds over the 2-skeleton $X^{(2)}$ of $X$, then \eqref{con} holds for $X$.
\end{lemma}

\begin{proof}
By standard results in cellular cohomology, the inclusion of the $2$-skeleton $\iota \colon X^{(2)} \hookrightarrow X$ induces an injection $\iota^* \colon H^2(X) \hookrightarrow H^2(X^{(2)})$.
The pairs $(E, G)$ and $(\hat E, \hat G)$ pull back to T-dual pairs $(E^{(2)}, G^{(2)})$ and $(\hat E^{(2)}, \hat G^{(2)})$ over $X^{(2)}$.
Writing $\pi^{(2)} \colon E^{(2)} \to X^{(2)}$, the map $\iota^*$ further induces an injection ${}^\pi E^{2,1}_{\infty} \hookrightarrow {}^{\pi^{(2)}}E^{2,1}_{\infty}$. 
We explained why this is true in the proof of Lemma \ref{restlemma}.
This fits into the diagram
\[
\begin{tikzcd}
\F^2 H^3(E) \arrow[r] \arrow[d] & \F^2 H^3(E^{(2)}) \arrow[d] \\
{}^\pi E^{2,1}_{\infty} \arrow[r, hook] & {}^{\pi^{(2)}}E^{2,1}_{\infty},
\end{tikzcd}
\]
which commutes by the naturality of the spectral sequence.
Consider $[G]$ in the top left, which is mapped to $[G^{(2)}]$ along the upper map, then mapped further to $[\sum_{i=1}^n y_i \otimes \iota^*\hat c_i ]$.
This latter statement is by assumption on $X^{(2)}$.
Injectivity of the lower horizontal map then implies that $[G]$ must be mapped to $[\sum_{i=1}^n y_i \otimes \hat c_i ]$ along the left map, as required. 
\end{proof}

\section{Uniqueness of T-duals}\label{sec:uniqueness}

To discuss uniqueness, we recall three notions of isomorphism from \cite{BunkeRumpfSchick}; isomorphisms of pairs $(E, G)$, isomorphisms of T-duality triples $\bigl((E, G), (\hat E, \hat G), G_J \bigr)$, and isomorphisms of extensions of a given pair to a T-duality triple.

\begin{definition}
An isomorphism between two pairs $(E, G)$ and $(E', G')$ is a principal
$T^n$-bundle isomorphism $f \colon E \to E'$ together with an isomorphism
$\psi\colon f^*G' \xrightarrow{\cong} G$.
\end{definition}

Although we have somewhat ignored it in the notation, we stress that all of the pairs we consider are over a fixed base space $X$ and that the bundle isomorphisms cover the identity map of $X$.
Something else to note about this definition is that, by considering a
non-trivial automorphism $f \colon E \to E$, it is possible for two pairs $(E,
G)$ and $(E, G')$ to be isomorphic even if $G$ and $G'$ are not isomorphic as
gerbes over $E$.

\begin{definition}
An isomorphism between T-duality triples $\bigl((E_1, G_1), (\hat E_1, \hat
G_1), G_J \bigr)$ and $\bigl( (E_2, G_2), (\hat E_2, \hat G_2), G_J' \bigr)$
is a pair $(f, \hat f)$ consisting of isomorphisms $f \colon E_1 \to E_2$ and
$\hat f \colon \hat E_1 \to \hat E_2$ together with an isomorphism $\psi\colon
F^*G_J \xrightarrow{\cong} G_J$, where $F$ is the canonical map induced by $f$ and $\hat f$ on the fiberwise joins.
\end{definition}

\begin{definition}
An isomorphism of two extensions of $(E, G)$ is an isomorphism of T-duality triples such that the automorphism $f \colon E \to E$ is the identity.
\end{definition}

The first uniqueness theorem answers the following question: If a pair $(E, G)$ has a T-dual, how unique is the T-dual bundle $\hat E$?
This was first answered in \cite{BunkeRumpfSchick}*{Theorem 2.24 (2)}.

\begin{theorem}\label{thm:uniqueness1}
If $(E, G)$ and $(\hat E, \hat G)$ are T-dual, then for every antisymmetric matrix $B \in \operatorname{Mat}(n,n,\Z)$ there exists a pair $(\hat E', \hat G')$ with $\hat c_i' = \hat c_i + \sum_{j=1}^n B_{ij}c_j$ that is also T-dual to $(E, G)$.
Moreover, every T-dual of $(E, G)$ satisfies this for some $B$.
\end{theorem}

\begin{proof}

Suppose that 
\begin{equation}\label{equalE2}
    \Big[ \sum_{i=1}^n y_i \otimes \hat c_i \Bigr]
    = \Bigl[ \sum_{i=1}^n y_i \otimes \hat c_i' \Bigr]
    \in {}^{\pi}E^{1,2}_\infty.
\end{equation}
This is true if and only if $\sum_{i=1}^n y_i \otimes \hat c_i - \sum_{i=1}^n y_i \otimes \hat c_i'$ is in the image of $d_2^{0,2}$ on the second page of the spectral sequence associated with $E$.
In other words, for some $\sum_{i,j} A_{ij} y_i \cup y_j \in H^2(T^n)$, we have
\begin{align*}
    \sum_{i=1}^n y_i \otimes (\hat c_i - \hat c_i')
    &= d_2^{0,2} \Bigl( \sum_{i,j} A_{ij} y_i \cup y_j \Bigr) \\
    &= \sum_{i,j} \bigl( A_{ij}y_i \otimes c_i - A_{ij}y_j \otimes c_i \bigr) \\
    &= \sum_{i=1}^n \sum_{j=1}^n y_i \otimes  (A_{ij} - A_{ji})c_j.
\end{align*}
We conclude that \eqref{equalE2} holds if and only if $\hat c_i - \hat c_i' = \sum_{j=1}^n B_{ij} c_j$ where $B_{ij} = A_{ij} - A_{ji}$ are the entries of an antisymmetric matrix.

Now, given $\hat E'$ with Chern classes $\hat c_i' = \hat c_i + \sum_{j=1}^n B_{ij}c_j$, we have 
\[
    [G]^{2,1} 
    =\Big[ \sum_{i=1}^n y_i \otimes \hat c_i \Bigr]
    = \Bigl[ \sum_{i=1}^n y_i \otimes \hat c_i' \Bigr],
\]
and so Lemma \ref{restlemma} implies that there exists $\hat G'$ making $(\hat E', \hat G')$ T-dual to $(E, G)$.

For the second statement, if $(\hat E, \hat G)$ and $(\hat E', \hat G')$ are both T-dual to $(E, G)$, then Theorem \ref{conditiontheorem} implies that
\[
    \Big[ \sum_{i=1}^n y_i \otimes \hat c_i \Bigr]
    =[G]^{2,1} 
    = \Bigl[ \sum_{i=1}^n y_i \otimes \hat c_i' \Bigr],
\]
and therefore $\hat c_i' = \hat c_i + \sum_{j=1}^n B_{ij}c_j$ for some antisymmetric matrix $B$.
\end{proof}

The following theorem, which is our version of \cite{BunkeRumpfSchick}*{Proposition 7.31}, is required to prove the next uniqueness result.
We provide a new proof of the result in Section \ref{sec:proof}.

\begin{theorem}\label{thm:pullbackThomclass}
Let $E$ and $\hat E$ be principal $T^n$-bundles and let $j\colon J \to X$ be their fiberwise join.
Given bundle automorphisms $\psi \colon E \to E$ and $\hat \psi \colon \hat E
\to \hat E$ let $\Psi \colon J \to J$ be the  canonical homeomorphism that restricts to $\psi$ and $\hat \psi$.

Then for any Thom class $T \in \thom(J)$, we have
\[
    \Psi^*T = T
        + j^*\Bigl(\sum_{i=1}^n c_i \cup \hat \psi_i 
        + \sum_{i=1}^n \hat c_i \cup \psi_i \Bigr),
\]
where $c_i$ and $\hat c_i$ are the Chern classes of $E$ and $\hat E$ and $\psi_i, \hat \psi_i \in H^1(X)$ are the classes associated with the automorphisms $\psi$ and $\hat \psi$.
\end{theorem}

\begin{corollary}\label{cor:pullbackThomclass}
Let $(E, G)$ and $(\hat E, \hat G)$ be T-dual and let $Q \to X$ be a gerbe.
If $[Q] = \sum_{i=1}^n c_i \cup \hat\psi_i$ for some $\hat\psi_1, \dotsc, \hat\psi_n \in H^1(X)$, then $(\hat E, \hat G)$ is isomorphic to $(\hat E, \hat G \otimes \hat\pi^*Q)$.
\end{corollary}

\begin{proof}
Let $\hat \psi \colon \hat E \to \hat E$ be the automorphism corresponding to the classes $\hat\psi_1, \dotsc, \hat\psi_n \in H^1(X)$.
Consider $\Psi \colon J \to J$ defined using $\psi = \id$ and the chosen $\hat \psi$.
Theorem \ref{thm:pullbackThomclass} implies that 
\[
    \Psi^*[G_J] = [G_J] + j^*\Bigl( \sum_{i=1} c_i \cup \hat\psi_i \Bigr).
\]
This implies that 
\[
    \hat\psi^*[\hat G] 
    = [\hat G] + \hat\pi^*\Bigl( \sum_{i=1}^n c_i \cup \hat\psi_i \Bigr) 
    = [\hat G] + \hat\pi^*[Q].
\]
Therefore $\hat\psi^*\hat G \cong \hat G \otimes \hat\pi^*Q$, which means that $(\hat E, \hat G)$ and $(\hat E, \hat G \otimes \hat\pi^*Q )$ are isomorphic.
\end{proof}

Theorem \ref{thm:existencetwobundles} implies that there is a transitive action of $H^3(X)$ on the isomorphism classes of T-duality triples connecting $E$ and $\hat E$.
Explicitly, if $\bigl( (E, G), (\hat E, \hat G), G_J \bigr)$ is a T-duality triple and $[Q] \in H^3(X)$, then 
\[
    \bigl((E, G \otimes \pi^*Q), (\hat E, \hat G \otimes \hat\pi^*Q), G_J \otimes j^*Q \bigr)
\]
is another T-duality triple containing $E$ and $\hat E$.
If $[Q] \in \ker(\pi^*)$, then this triple is also an extension of $(E, G)$.
If, in addition, $[Q] = \sum_{i=1}^n c_i \cup \hat \alpha_i$ for $\hat\alpha_i \in H^1(X)$, then Corollary \ref{cor:pullbackThomclass} implies that the two triples are isomorphic extensions of $(E, G)$.

With this in mind, consider the map
\[
    C \colon H^1(T^n) \otimes H^1(X) \to H^3(X)
    \quad \sum_{i=1} y_i \otimes \alpha_i \mapsto \sum_{i=1}^n c_i \cup \alpha_i.
\]
One may recognise this as the map $d^{1,1}_2$ in the spectral sequence of $E$.
It takes values in $\ker(\pi^*)$ because $\pi^*(c_i) = 0$ for each $c_i$.
By the previous discussion, we have an action of $\ker(\pi^*) / \im(C)$ on the isomorphism classes of extensions of $(E, G)$ whose dual bundle is $\hat E$.

The following theorem is \cite{BunkeRumpfSchick}*{Theorem 2.23(3)}.
We give a new and more direct proof of the key ingredients needed to establish
it.

\begin{theorem}\label{thm:uniqueness2}
Let $\bigl( (E, G), (\hat E, \hat G), G_J \bigr)$ be a T-duality triple.
The set of isomorphism classes of extensions of $(E, G)$ that have a dual bundle isomorphic to $\hat E$ is the orbit of $\bigl( (E, G), (\hat E, \hat G), G_J \bigr)$ under an effective action of $\ker(\pi^*) / \im(C)$.
\end{theorem}

\begin{proof}
Let $(\hat E, \hat G)$ and $(\hat E, \hat G')$ be both T-dual to $(E, G)$. 
By Theorem \ref{thm:existencetwobundles}, there exists $[Q] \in H^3(X)$ such that $\pi^*[Q] = 0$ and $[\hat G'] = [\hat G] + \hat\pi^*[Q]$.
Therefore $\ker(\pi^*)$, and hence $\ker(\pi^*) / \im(C)$, acts transitively on the set of pairs $(\hat E, \hat G')$ T-dual to $(E, G)$.

Now, suppose that $[Q] \in \ker(\pi^*)$ such that the extensions $\bigl( (E,
G), (\hat E, G), G_J\bigr)$ and $\bigl( (E, G), (\hat E, G \otimes \hat\pi^*
Q), G_J \otimes j^*Q\bigr)$ are isomorphic.
By definition, this means that there is an automorphism $\hat\psi \colon \hat E \to \hat E$ such that $\Psi^*[G_J] = [G_J] + j^*[Q]$, where $\Psi$ is constructed with $\hat \psi$ and the identity on $E$.
By Theorem \ref{thm:pullbackThomclass}, we then have
\[
    [G_J] + j^*[Q] 
        = \Psi^*[G_J] 
        = [G_J] + j^*\Bigl(\sum_{i=1} c_i \cup \hat \psi_i \Bigr).
\]
Since $j^*$ is injective, this implies that $[Q] = \sum_{i=1}^n c_i \cup \hat \psi_i \in \im(C)$.
Therefore, $\ker(\pi^*) / \im(C)$ acts freely on the set of extensions of $(E, G)$ with dual bundle $\hat E$. 
We have shown that $\ker(\pi^*) / \im(C)$ acts freely and transitively on the set of extensions of $(E, G)$ with a fixed dual bundle, so the proof is complete. 
\end{proof}

The following is now a corollary of Theorem \ref{thm:uniqueness1} and Theorem \ref{thm:uniqueness2}.

\begin{corollary}
In the $n=1$ case, T-duals are unique.
\end{corollary}

\begin{proof}
Let $(E, G)$ and $(\hat E, \hat G)$ be T-dual.
There are no non-zero antisymmetric matrices in $\operatorname{Mat}(1,1,\Z)$,
so Theorem \ref{thm:uniqueness1} implies that any other T-dual has $\hat E$ as
its underlying $S^1$-principal bundle.

In the $n=1$ case, we have the Gysin sequence
\[
    \dotsm 
    \to H^1(X)
    \xrightarrow{\, \cdot c \, } H^3(X)
    \xrightarrow{\, \pi^* \,} H^3(E)
    \to \dotsm,
\]
and therefore $\ker(\pi^*)$ consists of elements of the form $c \cup \alpha$ for $\alpha \in H^1(X)$.
This is precisely the image of $C$, so $\ker(\pi^*)/\im(C)$ is trivial. Therefore, Theorem \ref{thm:uniqueness2} now implies that $(\hat E, \hat G)$ is the unique T-dual of $(E, G)$, up to isomorphism.
\end{proof}

We have now proved the classification results of \cite{BunkeRumpfSchick}. 

\section{Proof of Theorem~\ref{thm:pullbackThomclass}}
\label{sec:proof}

Here we will give the proof of Theorem \ref{thm:pullbackThomclass}.
The proof is a series of reductions that will allow us to derive the result from the $X = S^2$ case, for which a geometric calculation is given.
Let us start by again stating the theorem.

\begin{theorem*}[Theorem \ref{thm:pullbackThomclass}]
Let $E$ and $\hat E$ be principal $T^n$-bundles and let $j\colon J \to X$ be
their fiberwise join.
Given bundle automorphisms $\psi \colon E \to E$ and $\hat \psi \colon \hat E
\to \hat E$ let $\Psi \colon J \to J$ be the  canonical homeomorphism that restricts to $\psi$ and $\hat \psi$.
For any Thom class $[G_J] \in \thom(J)$, we have
\begin{equation}\label{eqn:pullback}
    \Psi^*T = T
        + j^*\Bigl(\sum_{i=1}^n c_i \cup \hat \psi_i 
        + \sum_{i=1}^n \hat c_i \cup \psi_i \Bigr),
\end{equation}
where $c_i$ and $\hat c_i$ are the Chern classes of $E$ and $\hat E$ and $\psi_i, \hat \psi_i \in H^1(X)$ are the classes associated with the automorphisms $\psi$ and $\hat \psi$.
\end{theorem*}

The proof, written as follows, is intended to be just an overview for the reader to understand the main idea of the proof.
The main content of the proof is in the lemmas that come afterwards.

\begin{proof}[Proof of Theorem \ref{thm:pullbackThomclass}]
By Lemma \ref{lem:XtimesS1}, the theorem can be proved by determining the induced map $\alpha^* \colon H^3(J \times S^1) \to H^3(J \times S^1)$, where 
\[
    \alpha \colon J \times S^1 \to J \times S^1,
    \quad
    (p, z) \mapsto (p\cdot z, z),
\]
is defined using one of the $2n$ circle actions on $J$.

The goal then is to reduce to a calculation over $X = S^2$.
This is done in Lemma \ref{lem:reductionS2}.
There are several steps to the reduction.
First, we can reduce to 3-dimensional CW-complexes, and then to the $n=1$ case.
From here, we can reduce to a calculation over the 3-skeleton of $BS^1 \times BS^1$, which is $S^2 \vee S^2$.
$BS^1 \times BS^1$ is used because it is the universal space for pairs of circle bundles.
The calculation over $S^2 \vee S^2$ can then be done over each $S^2$ and patched together with the Mayer-Vietoris sequence.

To prove the $X=S^2$ case, it is more convenient to prove a corresponding result in homology, where we have explicit generators and can make a geometric argument.
The homology calculation is done in Lemma \ref{lem:homology}, where it is used in Lemma \ref{lem:S2} to get the final result for $S^2$.
At this point the proof is complete.
\end{proof}

Now, onto the details.
For ease of notation, we will write $c_1, \dotsc, c_n$ for the Chern classes of $E$ (as usual) and $c_{n+1}, \dotsc, c_{2n}$ for the Chern classes of $\hat E$.
We will further use the notation $\hat c_i = c_{i+n}$ for $i \in \{1, \dotsc n\}$ and $\hat c_i = c_{i-n}$ for $i \in \{n+1, \dotsc, 2n\}$.
This reflects the pairing between the $S^1$ factors of $E$ and $\hat E$ in the theorem: when we apply an automorphism to the $i$th circle factor of $E$, the pullback introduces a factor containing $\hat c_i = c_{i+n}$, whereas applying an automorphism on the $i$th factor of $\hat E$ results in a factor containing $\hat c_{i+n} = c_i$.

Note that $J$ has an action of $T^{2n}$ defined by $[e, \hat e, t] \cdot (z, \hat z) = [e \cdot z, \hat e \cdot \hat z, t]$.
By considering the $2n$ inclusions $S^1 \hookrightarrow T^{2n}$, one then has $2n$ possible $S^1$ actions on $J$.
It will be possible to consider each $S^1$ action individually:

\begin{lemma}\label{lem:XtimesS1}
Let $\alpha_i \colon J \times S^1 \to J \times S^1$ be defined as $\alpha(p,z) = (p\cdot z, z)$, where $S^1$ acts on the $i$th circle factor of $J$.
If $i \in \{1, \dotsc, 2n \}$ and $T$ is a Thom class for $J \times S^1$ that
has been pulled back from $J$, and if we have
\[
    \alpha^*T = T + j^*(\hat c_i) \otimes u,
\]
where $u \in H^1(S^1)$ is the generator, then Theorem \ref{thm:pullbackThomclass} holds.
\end{lemma}

\begin{proof}
    The automorphism $\Psi \colon J \to J$ can be factored as
\[
    J 
    \xrightarrow{\, (\id, j) \,}
    J \times X 
    \xrightarrow{\, (\id, (\psi, \hat\psi)) \,}
    J \times T^{2n} 
    \xrightarrow{\, m \,}
    J,
\]
where $m$ is the action of $T^{2n}$ on $J$.
It is then sufficient to show that
\[
    m^*T = T + \sum_{i=1}^{2n} j^*c_i \otimes y_i,
\]
where $y_1, \dotsc, y_{2n}$ are the generators of $H^1(T^{2n})$.
Indeed, if this holds then, as required,
\begin{align*}
    \Psi^*T 
        &= (\id, j)^*(\id, \psi, \hat\psi)^*m^*T \\
        &= (\id, j)^*(\id, \psi, \hat\psi)^*\bigl(T + \sum_{i=1}^{2n} j^*c_i \otimes y_i \bigr) \\
        &= (\id, j)^*\bigl(T + \sum_{i=1}^n j^*c_i \otimes \psi_i + \sum_{i=1}^n j^*\hat c_i \otimes \hat\psi_i \bigr) \\
        &= T + j^*\Bigl(\sum_{i=1}^n c_i \otimes \psi_i + \sum_{i=1}^n \hat c_i \otimes \hat\psi_i \Bigr).
\end{align*}
By additivity, one can consider a single $S^1$ action at a time, so it is sufficient to consider multiplication by a single $S^1$:
\[
    J \times S^1 
    \xrightarrow{\alpha_i}
    J \times S^1
    \xrightarrow{\, \pr_1 \,}
    J.
\]
The result now follows from the assumption made in the lemma.
\end{proof}

The following lemma details how we can reduce the proof to a calculation over $S^2$.

\begin{lemma}\label{lem:reductionS2}
If Theorem \ref{thm:pullbackThomclass} holds for $X = S^2$ equipped with the Hopf bundle and the trivial bundle, then it holds for all $X$.
\end{lemma}

\begin{proof}

By the previous lemma, we need to prove that if $\alpha \colon J \times S^1 \to J \times S^1$ is the action map for some choice of $S^1$ factor and $T \in H^3(J \times S^1)$ is pulled back from a Thom class on $J$, then
\[
    \alpha^*T = T + j^*(\hat c) \otimes u,
\]
where $\hat c$ is the `dual' of the Chern class corresponding to the chosen $S^1$-action.
Assuming that we have the result for $S^2$, we prove this with the following steps:
\begin{enumerate}

    \item Prove the $n=1$ result for $S^2 \vee S^2$ equipped with the canonical pair of $S^1$-bundles.

    \item Prove the result for the $n$-fold wedge product $\bigvee^n S^2 \vee S^2$ equipped with its canonical principal $T^n$-bundles.

    \item Prove the result for 3-dimensional CW-complexes.

    \item Obtain the result for general $X$ by restricting to its 3-skeleton.

\end{enumerate}

\textbf{Step 1:} 
Consider $S^2 \vee S^2$ with the pair of circle bundles classified by the two generators of $H^2(S^2\vee S^2) \cong H^2(S^2) \oplus H^2(S^2)$.
With $J$ as the fiberwise join of these bundles, the Mayer-Vietoris sequence
for the decomposition of $J$ into the restrictions to the two spheres gives an injective map
\[
0=H^2(J|_{\{*\}}\times S^1)\to    H^3(J \times S^1) 
    \hookrightarrow H^3(J|_{S^2} \times S^1) \oplus H^3(J|_{S^2} \times S^1).
\]
When the two bundles on $S^2 \vee S^2$ are restricted to a single sphere we obtain the Hopf bundle and the trivial bundle on $S^2$.
This implies that the $S^2 \vee S^2$ case follows the $S^2$ case.

\textbf{Step 2:}
Consider the $n$-fold wedge product $X = \bigvee^n S^2 \vee S^2$ equipped with the pair of $T^n$-bundles classified by the first and last $n$ generators of $H^2(X) \cong H^2(T^n) \oplus H^2(T^n)$, respectively.
The $n=1$ case was covered in Step 1.
The fiberwise join of the two bundles, which we denote by $J_n$, is isomorphic to a fiber product of $n$ copies of $J_1$:
\[
    J_n \cong J_1 \times_{X \times I} J_1 \times_{X \times I} \dotsm \times_{X \times I} J_1
\]
This is a specific example of the decomposition we have considered twice
already; see the discussion preceding Lemma \ref{restlemma}.

The $S^1$ action defining $\alpha_n \colon J_n \times S^1 \to J_n \times S^1$ can, without loss of generality, be chosen to be in the last $J_1$ factor of $J_n$.
Then $\alpha_n$ acts via the identity on all but the last $J_1$ factor.

The unique class in $H^3(J \times S^1)$ that has been pulled back from the Thom class for $J$ is of the form $T := \sum_{i=1}^n \pr_i^*T_1$, where $T_1 \in H^3(J_1 \times S^1)$ is the pullback of the unique Thom class on $J_1$.
These Thom classes are unique because wedges of two-spheres have trivial degree 3 cohomology.
We now have
\begin{align*}
    \alpha^*T 
    &= \alpha^*\Bigl( \sum_{i=1}^n \pr_i^*T_i \Bigr) \\
    &= \sum_{i=1}^{n-1} \pr_i^*T_1 + \alpha_1^*\pr_n^*T_1 \\
    &= \sum_{i=1}^{n-1} \pr_i^*T_1 + \pr_n^*T_1 + j^*(\hat c_n) \otimes u \\
    &= T + j^*(\hat c_n) \otimes u.
\end{align*}
Here, we have used the result for $n=1$, which we have from Step 1.
We conclude that the result holds for $\bigvee^{n}S^2 \vee S^2$ equipped with its canonical pair of $T^n$-bundles.

\textbf{Step 3:} 
Let $X$ be a 3-dimensional CW-complex equipped with two circle bundles $E$ and $\hat E$.
These are pulled back from the universal bundle $ET^n \to BT^n$ along some pair of maps $c,\hat c \colon X \to BT^n$.
Since $X$ is 3-dimensional, $c$ and $\hat c$ can be chosen so that the map $(c, \hat c)\colon X \to BT^n \times BT^n$ takes values in the 3-skeleton of $BT^n \times BT^n$, which is the same as the 2-skeleton $\bigvee^{n}S^2 \vee S^2$.
The corresponding $J$ then fits into the following pullback diagram:
\[
\begin{tikzcd}
    J 
        \rar["f"] \dar["j", swap]
    & J_n
        \dar["\pi"]
    \\
    X 
        \rar["{(c, \hat c)}"]
    & \bigvee^{n}S^2 \vee S^2
\end{tikzcd}
\]
Here, the bundle $J_n$ comes from the previous step.
Let $T_{univ} \in H^3(J_n \times S^1)$ be the pullback of the unique Thom class for $J_n$ and let $T' = (f,\id)^*T_{univ}$.
This is an element of $H^3(J \times S^1)$ that is pulled back from a Thom class for $J$.
We then calculate that
\begin{align*}
    \alpha^*T' 
    &= \alpha^*f^*T_{univ} \\
    &= f^*\alpha_{univ}^*T_{univ} \\
    &= f^*T_{univ} + f^*(\pi^*\hat c_{univ} \otimes u) \\
    &= T' + j^*(\hat c) \otimes u.
\end{align*}
We have written $\alpha_{univ}$ for the action map on $J_n$.
If $T \in H^3(J \times S^1)$ is a pullback of a Thom class for $J$, then it must be of the form $T = T' + (j, \id)^*\delta$ for some $\delta \in H^3(X \times S^1)$.
This is because any two Thom classes differ by a degree 3 cohomology class on the base space.
However, since $(j, \id) \circ \alpha = (j, \id)$, we have
\begin{align*}
    \alpha^*T 
    &= \alpha^*T' + \alpha^*(j,\id)^*\delta \\
    &= T' + j^*(\hat c) \otimes u + (j, \id)^*\delta \\
    &= T + j^*(\hat c) \otimes u,
\end{align*}
as required. 
This proves the result for 3-dimensional CW-complexes.

\textbf{Step 4:} 
Let $X^{(3)}$ be the 3-skeleton of a CW-complex $X$.
In this step, we show that if the theorem holds for $X^{(3)}$, which the previous step indeed implies, then the theorem holds for $X$ itself.
Letting $j_{(3)} \colon J^{(3)} \to X^{(3)}$ be the restriction of $J$ to $X^{(3)}$, we have the following commutative diagram with exact rows coming from \eqref{exactsequenceJ}:
\[
\begin{tikzcd}
    0 
        \rar
    & H^3(X)
        \rar["j^*"] \dar[hook]
    & H^3(J) 
        \rar \dar[hook]
    & H^3(T^n * T^n) 
        \dar["="]
    \\
    0 
        \rar
    & H^3(X^{(3)})
        \rar["j_{(3)}^*"]
    & H^3(J^{(3)}) 
        \rar
    & H^3(T^n * T^n) 
\end{tikzcd}
\]
By standard results in cellular cohomology, the vertical maps are injective.
If $T$ is a Thom class for $J$, then $\Psi^*T - T = j^*\delta$ for some difference element $\delta \in H^3(X)$.
Letting $T^{(3)}$ be the restriction of $T$ to $J^{(3)}$, we know by assumption that 
\[
    \Psi^*T^{(3)} - T^{(3)} = j_{(3)}^*\Bigl(\sum_{i=1}^n c_i \cup \hat \psi_i + \sum_{i=1}^n \hat c_i \cup \psi_i \Bigr)
\]
Injectivity of the left-most vertical map then implies that $\delta = \sum_{i=1}^n c_i \cup \hat \psi_i + \sum_{i=1}^n \hat c_i \cup \psi_i$, as required. 

With all the steps complete, we conclude that Theorem \ref{thm:pullbackThomclass} follows from the $S^2$ case.
\end{proof}

We now focus on the case $X=S^2$.
Let $j \colon J \to S^2$ be the fiberwise join of the Hopf circle bundle $H \to S^2$ and the trivial bundle $S^2 \times S^1 \to S^2$.
Elements of $J$ can be written as $[h, (x,z), t]$ with $h \in H$, $(x,z) \in S^2 \times S^1$, and $t \in [0,1]$. 
This represents the linear combination $t h + (1-t)(x,z)$ in the join of the two fibers over $x \in S^2$.
$J$ has two possible $S^1$-actions; one either acts on the fiber of the Hopf bundle or on the fiber of the trivial bundle. 
Choosing either of these actions gives the map $\alpha \colon J \times S^1 \to J\times S^1$.

To prove Theorem \ref{thm:pullbackThomclass} for $S^2$, we will determine the induced map $\alpha_*$ on the homology group $H_3(J \times S^1)$.
This is the subject of Lemma \ref{lem:homology}.
The calculation in homology was found to be easier because we have very explicit, geometric descriptions of the generators of $H_3(J \times S^1)$, which we discuss shortly.
Using the homology calculation, we then get the result in cohomology by passing to the dual.
This is done in Lemma \ref{lem:S2}.

The Leray-Serre spectral sequence for $J$ tells us that $H_3(J) \cong H_3(S^3)$ and $H_2(J) \cong H_2(S^2)$.
By the K\"unneth formula, we then have
\[
    H_3(J \times S^1) \cong H_3(S^3) \oplus H_2(S^2).
\]
The two generators of $H_3(J \times S^1)$ are $i_*[S^3]$, where $i \colon S^3 \hookrightarrow J$ is the inclusion of a fiber, and $\sigma_*[S^2 \times S^1]$, where $\sigma \colon S^2 \times S^1 \to J \times S^1$ is any section; the one we choose is $\sigma(x,\varphi) = \bigl([0, (x,1), 0], \varphi\bigr)$.

Elements of $H_3(J \times S^1)$ are all of the form $\kappa  a +\lambda b$, where $\kappa,\lambda \in \Z$, $a := i_*[S^3]$ and $b := \sigma_*[S^2 \times S^1]$.
Given such an element, one can determine $\kappa$ and $\lambda$ by taking the intersection product with $b$ and $a$, respectively.
Indeed, a fiber and the image of a section intersect transversely at a single point, so we have $a \cap b = 1$.
We further have $a \cap a = 0$ and $b \cap b = 0$, since both submanifolds
have a deformation disjoint from the submanifold.
Therefore,
\[
    (\kappa a + \lambda b) \cap b = \kappa
    \quad\text{and}\quad
    (\kappa a + \lambda b) \cap a = \lambda.
\]
Now, let us determine the induced map in homology.

\begin{lemma}\label{lem:homology}
The induced map $\alpha_* \colon H_3(J \times S^1) \to H_3(J \times S^1)$ is the identity when defined via the action on the Hopf bundle and
\[
    \alpha_*i_*[S^3] = [S^3],
    \quad
    \alpha_*\sigma_*[S^2 \times S^1] = \sigma^*[S^2 \times S^1] + i_*[S^3]
\]
when $\alpha$ is defined by acting on the trivial bundle.
\end{lemma}

\begin{proof}
First observe that $\alpha$ preserves fibers regardless of which action is used to define it, and so we have $\alpha_*i_*[S^3] = i_*[S^3]$.
We then only need to determine $\alpha_*\sigma_*[S^2 \times S^1]$ in each case.

If $\alpha$ is defined via the $S^1$-action on the Hopf bundle, then
\[
    \alpha \circ \sigma(x, \varphi)
    = \alpha\bigl( [0, [x, 1], 0], \varphi\bigr)
    = \bigl([0, [x,1], 0], \varphi\bigr)
    = \sigma(x,\varphi)
\]
So, $\alpha_*$ is just the identity in this case.

Now and for the remainder of the proof, let $\alpha$ be defined via the $S^1$-action on the trivial bundle.
The task is to identify the class $\alpha_*\sigma_*[S^2 \times S^1] \in H_3(J \times S^1)$.
As a submanifold, this consists of the points
\[
    \alpha \circ \sigma(S^2 \times S^1) 
    = \bigl\{ ([0, (x,\varphi), 0], \varphi) : (x, \varphi) \in S^2 \times S^1 \bigr\}.
\]
This transversely intersects with $i(S^3)$ and the intersection consists of a single point. 
We can then write $\alpha_*\sigma_*[S^2 \times S^1] = \lambda \sigma_*[S^2 \times S^1] + i_*[S^3]$ for some $\lambda \in \Z$.
To determine $\lambda$, we need to deform $\alpha \circ \sigma(S^2 \times S^1)$ so that it intersects transversely with $\sigma(S^2 \times S^1)$.

Let us write the sphere as the union of the north and south hemisphere, $S^2 = D^2_+ \cup D^2_-$, and elements of the disk $D^2$ with polar coordinates $x = r \omega$, where $r \in [0,1]$ and $\omega \in S^1$.
Of course, $\omega$ is not defined when $r=0$.
Define the following map:
\begin{align*}
    F \colon S^2 \times S^1 \times I &\longrightarrow J \times S^1, \\
    \bigl(x = r\omega \in D^2_+, \varphi, t \bigr) 
        &\longmapsto 
        \Bigl( 
            \bigl[ 
                (x, \omega \bigr), 
                (x, \varphi),
                t\cdot r
            \bigr], 
            \varphi 
        \Bigr),\\
    \bigl(x = r\omega \in D^2_-, \varphi, t\bigr) 
        &\longmapsto 
        \Bigl( 
            \bigl[ 
                (x, 1 \bigr), 
                (x, \varphi),
                t (1- r)
            \bigr], 
            \varphi 
        \Bigr)
\end{align*}
This is well-defined because if $x \in D^2_+ \cap D^2_-$ then $(x,z) \sim (x,1)$ in $H$, since this is the equivalence relation used to glue together $H|_{D^2_+} = D^2_+ \times S^1$ and $H|_{D^2_-} = D^2_- \times S^1$ to define $H$.
This map is a deformation from $\alpha \circ \sigma(S^2 \times S^1)$ to a new submanifold 
\[
    D := F\bigl(S^2 \times S^1 \times \{1\}\bigr).
\]
$D$ and $S := \sigma(S^2 \times S^1)$ intersect at a single point, $p := ([0, (x^+, 1), 0], 1)$, where $x^+$ is the north pole.
For this to be a transverse intersection we need that $T_pD + T_pS = T_p(J \times S^1)$.

We can restrict to $D^2_+$, noting that $J|_{D^2_+} \cong D^2_+ \times (S^1 * S^1)$.
We then have 
\[
    T(J|_{D_2^+} \times S^1) \cong TD^2 \oplus T(S^1 * S^1) \oplus TS^1.
\]
Observe that 
\[
    \sigma(D^2_+ \times S^1) \cong D^2 \times \bigl\{[0,1,0]\bigr\} \times S^1 \subseteq D^2 \times (S^1 * S^1) \times S^1.
\]
Therefore, over $J|_{D^2} \times S^1$ we have 
\[
    TS \cong TD^2 \oplus 0 \oplus TS^1 \subseteq TD^2 \oplus T(S^1 * S^1) \oplus TS^1.
\]
This says that the tangent space of $S$ fills the $D^2$ and $S^1$ directions, but is complementary to the tangent space of $S^1 *S^1$.
It now suffices to show that $T(S^1 * S^1) \subseteq TD$.

For this, consider the composition
\begin{gather*}
    D^2_+ \times S^1 
    \xrightarrow{\, F_1 \,}
    D^2_+ \times (S^1 * S^1) \times S^1
    \to
    S^1 * S^1 \\
    (r\omega, \varphi) \longmapsto [\omega, \varphi, t].
\end{gather*}
One can check that this is an embedding, for example by considering a diffeomorphism $S^1 * S^1 \cong H$ such that the above composition is a diffeomorphism between $D^2_+ \times S^1$ and $H|_{D^2_+} = D^2_+ \times S^1$.
This implies that $T_pD$ contains the tangent space at $S^1 * S^1$, as required.

We conclude that one can deform $\alpha \circ \sigma(S^2 \times S^1)$so that it transversely intersects $\sigma(S^2 \times S^1)$ at a single point.
Therefore,
\[
    \alpha_*\sigma_*[S^2 \times S^1] = \sigma_*[S^2 \times S^1] + i_*[S^3] \in H_3(J \times S^1).
\]
This completes the proof.
\end{proof}

\begin{lemma}\label{lem:S2}
Theorem \ref{thm:pullbackThomclass} holds for $S^2$ equipped with the Hopf bundle $H \to S^2$ and the trivial bundle $S^2 \times S^1 \to S^1$.
\end{lemma}

\begin{proof}
The universal coefficient theorem implies that
\begin{equation}\label{eqn:dual}
    H^3(J \times S^1) \cong \Hom\bigl(H_3(J \times S^1), \Z\bigr).
\end{equation}
$H^3(J \times S^1)$ has two generators; one is the pullback of a Thom class on $J$, and the other is $j^*(c_{H}) \otimes u$, where $c_{H} \in H^2(S^2)$ is the generator (also the Chern class of the Hopf bundle).
The right-hand side also has two generators, $a^*$ and $b^*$ defined by
\[
a^*(i_*[S^3]) = 1, \quad a^*(\sigma_*[S^2 \times S^1]) = 0,
\]\[
b^*(i_*[S^3]) = 0, \quad b^*(\sigma_*[S^2 \times S^1]) = 1.
\]
Under \eqref{eqn:dual}, the Thom class maps to $a^*$ and $j^*(c) \otimes u$ maps to $b^*$.

The map that $\alpha$ induces on $H_3(J \times S^1)$ induces a map on $\Hom(H_3(J \times S^1), \Z)$ that is compatible via \eqref{eqn:dual} with the map $\alpha$ induces on cohomology.
Therefore the map on homology, which we know from Lemma \ref{lem:homology}, determines the map on cohomology.

When $\alpha$ is defined via the action on the Hopf bundle, $\alpha_*$ is the identity.
Therefore the induced map on cohomology is also the identity.
This is what we desire in this case because the Chern class of the trivial bundle is zero.
When $\alpha$ is defined via the action on the trivial bundle, we have
\[
    a^* \bigl( \alpha_*i_*[S^3] \bigr)
    =  a^*\bigl(i_*[S^3]\bigr) = 1
    \quad\text{and}\quad
\]
\[
    a^* \bigl(\alpha_*\sigma_*[S^2 \times S^1] \bigr)
    = a^*\bigl(\sigma_*[S^2 \times S^1] + i_*[S^3]\bigr)
    = 1.
\]
So $\alpha^*(a^*) = a^* + b^*$.
Therefore, if $T \in H^3(J \times S^1)$ is pulled back from a Thom class for $J$, then
\[
    \alpha^*(T) = T + j^*(c_H) \otimes u.
\]
This completes the proof.
\end{proof}

We finish the section with a comparison to the original proof by Bunke, Rumpf, and Schick \cite{BunkeRumpfSchick}*{Proposition 7.31}.
In their paper, they construct a classifying space $R_n$ for T-duality triples so that the data of a T-duality triple (including the bundles and twists) is stored in a map $X \to R_n$.

A pair of bundles $E$ and $\hat E$ with automorphisms $\psi$ and $\hat \psi$  are classified by a map $H \colon X \times [0,1] \to K(\Z^n,2) \times K(\Z^n,2)$.
Consider the following diagram:
\[
\begin{tikzcd}
    X
        \rar["f"] \dar["i_0"]
    & R_n 
        \dar \\
    X \times [0,1]
        \rar["H"] \ar[ru, "\tilde H"]
    & K(\Z^n, 2) \times K(\Z^n,2)
\end{tikzcd}
\]
The authors start with a T-duality triple; let $f$ be its classifying map.
The lift $\tilde H$ essentially defines a homotopy between the given T-duality triple and the pullback of said triple along the automorphisms $\psi$ and $\hat \psi$.
By general theory about the classification of lifts, there is an element $X \to K(\Z^n, 3)$ measuring the difference between the two maps into $R_n$.
Using some results stated (but not proved) in the appendix, they have a specific description of this difference element and manage to show that it is homotopic to the map classifying $\sum_{i=1}^n c_i \cup \hat\psi_i + \sum_{i=1}^n \hat c_i \cup \psi_i$, as required.

The advantage and disadvantage of such homotopy theoretic proofs is that they are abstract and hide a lot of the underlying structure of the objects being worked with, in this case T-duality triples.
In comparison, our proof uses explicit calculations on the effect of the automorphisms $\psi$ and $\hat \psi$ on the Thom class.
It is more transparent in the sense that we are working directly with the objects being studied in the proof instead of homotopy classes of maps that classify them.

\section{A local definition of T-duality}
\label{sec:localtduality}

A major motivation for our definition of T-duality is to have a definition without any global assumptions on the H-flux. 
This allows a local definition of T-duality, like in \cite{BSS}. 
In their paper, a T-duality diagram (for circle bundles) is a diagram that is locally isomorphic to the universal T-duality diagram:
\begin{equation}\label{universal}
\begin{tikzcd}[column sep={4.5em,between origins},row sep=1.5em]
& 0 \arrow[ld] \arrow[rd] \arrow[rr, "u"] & & (\pi,\id)^*\hat G \arrow[ld] \arrow[rd] & \\
0 \arrow[rd] & & ES^1 \times S^1 \arrow[ld] \arrow[rd, "{(\pi,\id)}"] & & \hat G \arrow[ld] \\
& ES^1 \arrow[rd, "\pi"] & & BS^1 \times S^1 \arrow[ld] & \\
& & BS^1 & &
\end{tikzcd}
\end{equation}
Here $\pi \colon ES^1 \to BS^1$ is the universal principal $S^1$-bundle and
$\hat G$ is the unique (up to isomorphism) gerbe with Dixmier-Douady class $x \otimes \hat y \in H^3(BS^1 \times S^1)$, where $x$ and $\hat y$ are the generators of $H^2(BS^1)$ and $H^1(S^1)$, respectively.
With the global H-flux assumptions disposed of, we can generalise this definition to T-duality for principal $T^n$-bundles.
Consider the $n$-fold product of diagram \eqref{universal}:
\begin{equation}\label{universaltorus}
\begin{tikzcd}[column sep={4.5em,between origins},row sep=1.5em]
    & 0 
        \arrow[ld] \arrow[rd] \arrow[rr, "u_n"] 
    && (\pi,\id)^{n*} G_n 
        \arrow[ld] \arrow[rd]  
    & \\
    0 
        \arrow[rd] 
    & & ET^n \times T^n 
        \arrow[ld] \arrow[rd, "{(\pi,\id)^n}"] 
    & & G_n 
        \arrow[ld] \\
    & ET^n 
        \arrow[rd, "\pi^n"] 
    & & BT^n \times T^n
        \arrow[ld] 
    & \\
    & & BT^n & &
\end{tikzcd}
\end{equation}
Note the three types of products: Cartesian products of spaces, products of gerbes, and the product of gerbe isomorphisms.
We have $[G_n] = \sum_{i=1}^n x_i \otimes \hat y_i$, where $x_i$ and $\hat y_i$ are the generators of $H^2(BT^n)$ and $H^1(T^n)$, respectively.
This is the universal T-duality diagram for torus T-duality.
We have a third definition of T-duality:

\begin{definition}\label{localdefn}
A \emph{T-duality diagram} is a diagram of the form \eqref{Tdiagram} such that for each $x \in X$ there exists an open neighbourhood $U$ of $x$ together with a map $f \colon U \to BT^n$ such that the restriction of \eqref{Tdiagram} is isomorphic to the pullback of the universal diagram along $f$.
\end{definition}

This definition can be easily adapted to a definition for stacks, as was done in \cite{BSS}.
The following result is a generalisation of \cite{BSS}*{Lemma 4.1.5} to torus T-duality over spaces.
We assume that our base spaces have the homotopy type of a CW complex.

\begin{theorem}
Definition \ref{localdefn} is equivalent to Definition \ref{defn}
\end{theorem}
\begin{proof}
See the proof in \cite{BSS}*{Lemma 4.1.5}, in particular the first paragraph.
The authors intend to prove that their definition is equivalent to the one in \cite{BunkeRumpfSchick} but do not mention the global conditions on the flux.
We now know that these conditions are not necessary, so the proof is complete (and indeed holds in the torus case).
\end{proof}

\section{The T-duality transformation}\label{sec:Ttransformation}

T-duality is a relation between two space-time models that may appear very different but are in fact physically equivalent.
This implies a correspondence between the topological information describing a model and its T-dual.
In particular, there must be a correspondence between any cohomology groups that carry physical information about the space-time model. 
For this reason, any formulation of topological T-duality must come equipped with a T-duality transformation:
Given two T-dual pairs $(E, G)$ and $(\hat E, \hat P)$, there is an isomorphism
\[
    K^*(E, G) \cong K^{*-n}(\hat E, \hat G)
\]
between the $G$-twisted K-theory of $E$ and the $\hat G$-twisted K-theory of $\hat E$.
In this section, we will define a T-duality transformation using the Thom class definition of T-duality triples (Definition \ref{def:ThomTduality}).

Let $\bigl( (E, G), (\hat E, \hat G), G_J \bigr)$ be a T-duality triple, as per Definition \ref{def:ThomTduality}.
By Proposition \ref{prop:bijection}, there is a gerbe isomorphism $u\colon p^*P \to \hat p^* \hat P$ corresponding to $G_J$.
As a reminder, this is obtained by using a homotopy between the two maps
\[
    E \times_X \hat E \to E \hookrightarrow J
    \quad \text{and} \quad
    E \times_X \hat E \to \hat E \hookrightarrow J,
\]
noting that the pullback of $G_J$ along these two maps is $p^*G$ and $\hat p^* \hat G$, respectively.

\begin{definition}
Let $\bigl( (E, G), (\hat E, \hat G), G_J \bigr)$ be a T-duality triple.
The T-duality transformation is the composition
\[
    K^*(E, G)
    \xrightarrow{p^*} K^*(E \times_X \hat E, p^*G)
    \xrightarrow{u^*} K^*(E \times_X \hat E, \hat p^*\hat G)
    \xrightarrow{\hat p_!} K^{*-n}(\hat E, \hat G),
\]
where $u^*$ is the isomorphism induced by the twist morphism $u \colon p^*G \to \hat p^* \hat G$ corresponding to $G_J$ and $\hat p_!$ is the push-forward along $\hat p \colon E \times_X \hat E \to \hat E$.
\end{definition}

It is clear that this definition agrees with the T-duality transformation defined in \cite{BunkeRumpfSchick}*{Definition 3.12}, because the correspondence between gerbes $G_J$ and twist morphisms $u \colon p^*P \to \hat p^*\hat P$ is precisely the one used to show that the two definitions of T-duality are equivalent. 
Therefore, the following theorem is a result of the same statement being true in \cite{BunkeRumpfSchick}*{Theorem 3.13}.

\begin{theorem}
If $X$ is a finite CW-complex, then the T-duality transformation is an isomorphism.
\end{theorem}

An interesting application is the following result:

\begin{corollary}
Let $\pi \colon E \to X$ be a principal $T^n$-bundle with Chern classes $c_1, \dotsc, c_n \in H^2(X)$.
Then there is an isomorphism
\[
    K^*(E) \cong K^{*-n}(X \times T^n, G), 
\]
where $G$ is the gerbe with Dixmier-Douady class 
\[
    \sum_{i=1}^n c_i \otimes y_i \in H^2(X) \otimes H^1(T^n) \hookrightarrow H^3(X \times T^n),
\]
with $y_1, \dotsc, y_n$ being the generators of $H^2(T^n)$.
\end{corollary}

\begin{proof}
This follows from $(E,0)$ being T-dual to $(X \times T^n, G)$.
\end{proof}

We finish with some examples:

\begin{example}
Let $\Sigma_g$ denote the surface of genus $g$.
There is a canonical circle bundle $E_g \to \Sigma_g$ consisting of unit tangent vectors. We calculate its K-theory.
By the T-duality isomorphism, we have
\[
  K^*(E_g) \cong K^{*-1}(
  \Sigma_g \times S^1, \chi_g).
\]
Here, we note that $H^3(\Sigma_g \times S^1) \cong H^2(S_g) \cong \Z$, so we write our twists as integers, and $\chi_g := \chi_g(\Sigma_g) = 2-2g$ is the Euler characteristic of the surface. 
This appears as the twist because the Euler class is equal to the top Chern class, which in this case is the first Chern class of $E_g$.

We can calculate $K^{*-1}(\Sigma_g \times S^1, \chi_g)$ using the Atiyah-Hirzebruch spectral sequence.
The sequence stabilises after the third page, where the differential is given by $\operatorname{Sq}^3 - \chi_g$, where $\operatorname{Sq}^3$ is the third Steenrod operation. 
$\operatorname{Sq}^3$ factors as
\[
    H^0(\Sigma_g \times S^1;\Z)
    \to H^0(\Sigma_g \times S^1; \Z_2)
    \xrightarrow{\operatorname{Sq}^2} H^2(\Sigma_g \times S^1; \Z_2)
    \to H^3(\Sigma_g \times S^1; \Z),
\]
where the first map is reduction mod 2 and the last map is the Bockstein homomorphism.
This final map is actually $\Z_2^{2g+1} \to \Z$, which is necessarily trivial.
Therefore, $\operatorname{Sq}^3$ is the trivial map and the non-trivial differentials on the third page are given by multiplication by $\chi_g$.
The spectral sequence is easily shown to yield
\[
  K^0(
  \Sigma_g \times S^1, \chi_g) \cong \Z/|\chi_g|\Z
    \,\, \text{and} \,\,
    K^1(\Sigma_g \times S^1, \chi_g) \cong \Z^{2g+1}.
\]
Therefore,
\[
    K^0(E_g) \cong \Z^{2g+1}
    \,\, \text{and} \,\,
    K^1(E_g) \cong \Z/|\chi_g|\Z.
\]
More generally, the T-duality isomorphism applies to 3-manifolds that are circle bundles over surfaces.
If $E_{g,k} \to \Sigma_g$ is the circle bundle with Chern class $k \in \Z \cong H^2(\Sigma_g;\Z)$, then
\[
    K^*(E_{g,k}, \ell) \cong K^{*-1}(E_{g,\ell}, k).
\]
We write the twists as integers because the degree 3 cohomology of a connected oriented
3-manifold is identified with $\Z$.
\end{example}

\begin{example}

Let $(M, \eta)$ be a compact regular contact manifold of dimension $2n+1$.
By the Boothby-Wang theorem \cite{Blair}*{\textsection 3.3}, $M$ is the total space of a principal $S^1$-bundle over a symplectic manifold $(N, \omega)$, where the symplectic form $\omega$ is classified by an integral cohomology class; indeed, it is the class classifying the principal bundle.
The T-duality isomorphism gives
\[
    K^*(M) \cong K^{*-1}(N \times S^1, \tau),
\]
where $\tau$ is a twist classified by $[\pr_1^*\omega] \in H^3(N \times S^1 ; \Z)$.
Thus we have a relationship between the K-theory of a compact regular contact manifold and its associated symplectic manifold.
\end{example}

\begin{example}
We present an example of how the non-uniqueness of T-duals could give interesting isomorphisms.
Suppose that $(E, G)$ and $(E, G')$ are both T-dual to $(\hat E, \hat G)$, then
\[
    K^*(E, G) 
    \cong K^{*-1}(\hat E, \hat G)
    \cong K^*(E, G').
\]
By Theorem \ref{thm:uniqueness2}, we know that this only occurs when $G' \cong G \otimes \pi^*Q$, for some $[Q] \in \ker(\hat \pi^*)$. 
We find that
\[
    K^*(E, G) \cong K^*(E, G \otimes \pi^*Q).
\]
for all $[Q] \in \ker(\hat \pi^*)$.
So, we have identified an orbit of twists that produce isomorphic twisted K-groups.
\end{example}

% REFERENCES -----------------------

\end{document}